\documentclass[12pt,thmsa]{article}
\usepackage{amssymb}
\usepackage{amsfonts}
\usepackage{cite}
\usepackage{amsmath}
\usepackage[top=3.8cm,bottom=3.8cm,textwidth=16cm,centering,a4paper]{geometry}
\usepackage{color}

\numberwithin{equation}{section}
\newtheorem{theorem}{Theorem}[section]

\newtheorem{definition}[theorem]{Definition}

\newtheorem{lemma}[theorem]{Lemma}

\newtheorem{proposition}[theorem]{Proposition}
\newtheorem{remark}[theorem]{Remark}

\newenvironment{proof}[1][Proof]{\noindent\textbf{#1.} }{\hfill $\square$}

\begin{document}

\title{On the nonlinear Schr\"{o}dinger--Poisson system with
positron--electron interaction}
\date{}
\author{Ching-yu Chen, Yueh-cheng Kuo and Tsung-fang Wu\thanks{%
Corresponding author. E-mail address: tfwu@nuk.edu.tw(T.-F. Wu)} \\
{\footnotesize \emph{Department of Applied Mathematics, National University
of Kaohsiung, Kaohsiung 811, Taiwan }}}
\maketitle

\begin{abstract}
We study the Schr\"{o}dinger--Poisson type system:
\begin{equation*}
\left\{
\begin{array}{ll}
-\Delta u+\lambda u+\left( \mu _{11}\phi _{u}-\mu _{12}\phi _{v}\right) u=%
\frac{1}{2\pi }\int_{0}^{2\pi }\left\vert u+e^{i\theta }v\right\vert
^{p-1}\left( u+e^{i\theta }v\right) d\theta & \text{ in }\mathbb{R}^{3}, \\
-\Delta v+\lambda v+\left( \mu _{22}\phi _{v}-\mu _{12}\phi _{u}\right) v=%
\frac{1}{2\pi }\int_{0}^{2\pi }\left\vert v+e^{i\theta }u\right\vert
^{p-1}\left( v+e^{i\theta }u\right) d\theta & \text{ in }\mathbb{R}^{3},%
\end{array}%
\right.
\end{equation*}%
where $1<p<3$ with parameters $\lambda ,\mu_{ij}>0$. Novel approaches are employed to prove the existence of
a positive solution for $1<p<3$ including, particularly, the finding of a ground state solution for $2\leq p<3$ using established linear algebra techniques and demonstrating the existence of two distinct positive solutions for $1<p<2.$  The analysis here, by employing alternative techniques, yields additional and improved results to those obtained in the study of Jin and Seok [Calc. Var. (2023) 62:72].
\vspace{2mm}


\end{abstract}

\section{Introduction}

In this paper, we study the Schr\"{o}dinger--Poisson type systems:%
\begin{equation}
\left\{
\begin{array}{ll}
-\Delta u+\lambda u+\left( \mu _{11}\phi _{u}-\mu _{1,2}\phi _{v}\right) u=%
\displaystyle\frac{1}{2\pi }\int_{0}^{2\pi }\left\vert u+e^{i\theta
}v\right\vert ^{p-1}\left( u+e^{i\theta }v\right) d\theta & \text{ in }%
\mathbb{R}^{3}, \\
-\Delta v+\lambda v+\left( \mu _{22}\phi _{v}-\mu _{1,2}\phi _{u}\right) v=%
\displaystyle\frac{1}{2\pi }\int_{0}^{2\pi }\left\vert v+e^{i\theta
}u\right\vert ^{p-1}\left( v+e^{i\theta }u\right) d\theta & \text{ in }%
\mathbb{R}^{3},%
\end{array}%
\right.  \tag*{$\left( E\right) $}
\end{equation}%
where $u,v:\mathbb{R}^{3}\rightarrow \mathbb{R},\,\lambda, \mu_{ij}>0$ for $%
i, j=1, 2$ and $1<p<3$ with the function $\phi _{w}\in D^{1,2}(\mathbb{R}%
^{3})$ given by
\begin{equation}
\phi _{w}(x)=\int_{\mathbb{R}^{3}}\frac{w^{2}(y)}{|x-y|}dy.  \label{1-2}
\end{equation}%
Such equation is variational and its solutions are critical points of the
corresponding energy functional $J:H\rightarrow \mathbb{R}$ defined as%
\begin{eqnarray*}
J(u,v) &=&\frac{1}{2}\int_{\mathbb{R}^{3}}|\nabla u|^{2}+\lambda u^{2}dx+%
\frac{1}{2}\int_{\mathbb{R}^{3}}|\nabla v|^{2}+\lambda v^{2}dx \\
&&+\frac{1}{4}\int_{\mathbb{R}^{3}}\mu _{11}\phi _{_{u}}u^{2}+\mu _{22}\phi
_{_{v}}v^{2}-2\mu _{12}\phi _{_{v}}u^{2}dx \\
&&-\frac{1}{2\pi \left( p+1\right) }\int_{\mathbb{R}^{3}}\int_{0}^{2\pi
}\left\vert u+e^{i\theta }v\right\vert ^{p+1}d\theta dx,
\end{eqnarray*}%
where $H:=H^{1}\left(\mathbb{R}^{3}\right)\times H^{1}\left(\mathbb{R}^{3}\right). $ Note that $(u,v)\in H$ is a solution of system $\left( E\right) $ if and
only if $\left( u,v\right) $ is a critical point of $J.$ The couple $(u,v)$
is called a ground state solution of system $(E)$, if $\left( u,v\right) $
is a solution of the system and a minimum among all nontrivial solutions.

The system $(E)$ stems from the study of the nonlinear Maxwell-Klein-Gordon
equation in the limit of infinite light speed where the decomposition of the
wave functions results in the following system
\begin{equation*}
\left\{
\begin{array}{ll}
2i\,\dot{v}_{+}-\Delta v_{+}+\left( \mu _{11}\phi _{v_{+}}-\mu _{12}\phi
_{v_{-}}\right) v_{+}=\displaystyle\frac{1}{2\pi }\displaystyle%
\int_{0}^{2\pi }g(v_{+}+e^{i\theta }\bar{v}_{-})d\theta , &  \\
2i\,\dot{v}_{-}-\Delta v_{-}+\left( \mu _{22}\phi _{v_{-}}-\mu _{12}\phi
_{v_{+}}\right) v_{-}=\displaystyle\frac{1}{2\pi }\displaystyle%
\int_{0}^{2\pi }g(v_{-}+e^{i\theta }\bar{v}_{+})d\theta , &
\end{array}%
\right.
\end{equation*}%
with $v_{+}$ and $v_{-}$ being the decomposed positron and electron part of
the wave solutions respectively and $g$ the potential. For more detailed
description on the physical background and the derivation of the system, we
refer the interested readers to the paper by Jin and Seok\cite{Jin} and the
references therein. Further assumptions of separable forms of solutions for $%
v_{+}$ and $v_{-}$, namely,
\begin{equation*}
v_{+}=u(x)e^{i\frac{\lambda }{2}t}\quad \mathrm{and}\quad v_{-}=v(x)e^{i%
\frac{\lambda }{2}t},
\end{equation*}%
give rise to system $(E)$ where a standard power function of $g(u)=|u|^{p-1}u$
is assumed. The system has been carefully studied by Jin and Seok in \cite%
{Jin} and since the focus of our study is to extend and improve on their
analysis, for the paper the be self-contained, a brief account of their
results will be given below; but first we define some concepts of triviality
and positiveness of a vector function $\left( u,v\right) .$

\begin{definition}
A vector function $\left( u,v\right) $ is said to be\newline
$\left( i\right) $ nontrivial if either $u\neq 0$ or $v\neq 0;$\newline
$\left( ii\right) $ semi-trivial if it is nontrivial but either $u=0$ or $%
v=0;$\newline
$\left( iii\right) $ vectorial if both of $u$ and $v$ are not zero;\newline
$\left( iv\right) $ nonnegative if $u\geq 0$ and $v\geq 0;$\newline
$\left( v\right) $ positive if $u>0$ and $v>0.$
\end{definition}

Jin and Seok in \cite{Jin} considered two cases of system $(E)$, namely,
when the potential function $g(u)$ is set to zero and when $g(u)=|u|^{p-1}u$, and put
their results of the coupled system in comparison respectively with those of
a Hartree equation,
\begin{equation}
-\Delta u +\lambda u+ \mu \phi _{u} u = 0,  \label{eq:sp0}
\end{equation}
when $g=0$ and with those of a single nonlinear Schr\"{o}dinger--Poisson
equation , i.e.
\begin{equation}
-\Delta u +\lambda u+ \mu \phi _{u} u = |u|^{p-1} u,  \label{eq:sp1}
\end{equation}
when $g(u)= |u|^{p-1} u$ is the given power function. These equations appear in the study of semiconductor theory and have
been investigated by many; see for example \cite{BF,Lions,A,AP,CW,R1,R2,SWF,SWF1,SWF2,Wu,ZZ}. The nonlinear term $|u|^{p-1}u$ (or a more general form of $g(u)$) has been used conventionally in the Schr\"{o}dinger-Poisson equation to model the interaction among particles (possibly nonradial). It is known that equation (\ref{eq:sp0}) admits a unique radial solution when $\mu<0$ and, if $\mu\ge 0,$ only the
trivial solution is permitted. Jin and Seok\cite{Jin} likened the role of $%
\mu$ of equation (\ref{eq:sp0}) to that of det$(\mu_{ij})$ (i.e. $\,\mu_{11}\mu_{22}-\mu_{12}^2$) in system $(E)$
when the RHS vanishes, as they demonstrated that system ($E$) similarly
admits only the trivial solution when det$(\mu_{ij})\ge 0$ and a unique
positive vector solution exists if det$(\mu_{ij})< 0$.

With the standard power functions assumed on the RHS of equation (\ref%
{eq:sp1}), the solution structure varies within different range of $p$. In
the sub-linear and super-critical range of $p\in (0,1)\cup [5,\infty)$,
there is no nontrivial solution for equation (\ref{eq:sp1}); similarly, Jin
and Seok\cite{Jin} proved that only the trivial solution is permitted for
system ($E$) but unlike in equation (\ref{eq:sp1}), additional conditions on
$\mu_{ij}$, i.e. det($\mu_{ij})\ge 0$ is imposed.

In the range of $p\in (1,5)$, as with the single Schr\"{o}dinger--Poisson
equation case, the solution structure of system ($E$) changes at around $p=2$. %
Jin and Seok \cite{Jin} proved their results by considering the Mose index of the critical point of the energy functional and the existence of solutions were obtained subject to various additional
conditions, for clarity, we itemise their results below.

\begin{itemize}
\item[$\left( i\right) $] For $1<p\leq 2$ and $\lambda \geq 2,\mu _{ij}>0$
for $i,j=1,2,\, $ system ($E$) permits only the trivial solution when $\mu
_{11}>4$ and $\left( \mu _{11}-4\right) \left( \mu _{22}-4\right) >\mu
_{12}^{2}$.

\item[$\left( ii\right) $] For $1<p\leq 2$ and $\lambda ,\mu _{11},\mu
_{22}>0$ fixed. A positive solution exists for system ($E$) provided $\mu
_{12}>\mu _{0}$ for some constant $\mu _{0}>0$ (i.e. $\mu _{11}\mu _{22}-\mu
_{12}^{2}<0$).

\item[$\left( iii\right) $] For $1<p<2$ and $\lambda ,\mu _{ij}>0$ for $%
i,j=1,2$. At least two positive radial solutions are permitted for system ($E
$) when $\mu _{11}\mu _{22}-\mu _{12}^{2}>0$ and $\mu _{ij}$ is sufficiently
small, where $i,j=1,2.$

\item[$\left( iv\right) $] For $\frac{\sqrt{73}-2}{3}\leq p<5$ and parameters
$\lambda ,\mu _{ij}>0$ for $i,j=1,2$, system ($E$) has a positive ground
state solution. If $\mu _{11}=\mu _{22}$ is imposed, then
the range for the existence of a positive ground state solution is extended to $2<p<5$.
\end{itemize}


\noindent Result $(i)$ indicates that, for $1<p\leq 2$, a positive solution must lie within the L-shaped region of $0<\mu_{11}<4$ or $0<\mu_{22}<4$ and from result $(ii)$ the existence of a positive solution is only permitted if  $\mu_{12}$ is sufficiently large, namely, $\mu _{11}\mu _{22}-\mu_{12}^{2}<0.$
For the existence of a positive ground state solution in $(iv)$, the result falls short of $2<p<5$ unless a much stronger constraint of $\mu _{11}=\mu _{22}$ is imposed. Jin and Seok \cite{Jin}, however, suggested that the shorter range of $\frac{\sqrt{73}-2}{3}\leq p<5$ instead of $2<p<5$ and the constraint of $\mu _{11}=\mu _{22}$ are likely to be technical issues. Questions naturally arise as to whether we can improve on these results by demonstrating the existence of positive (ground state) solutions under weaker conditions of the parameter values of $\lambda$ and $\mu_{ij}$'s.

We begin by looking for solutions in the L-shaped region of $\mu_{11}$ and $\mu_{22}$ without the intension of imposing any constraint on $\mu_{12}$. As a result, we are able to identify a threshold number $\Lambda_0$ for the parameter values of $\mu_{11}$ and $\mu_{22}$ such that a positive solution is always permitted provided $\min\{\mu_{11},\mu_{22}\}<\Lambda_0$, as stated in Theorem \ref{tm1.2} below. In Theorem \ref{tm1.3},  the existence of a positive ground state solution is extended to $2\le p<\frac{\sqrt{73}-2}{3}$
including $p=2$ where a lesser constraint on the parameter values is required and, similarly, in Theorem \ref{tm1.4}, no additional condition on $\mu_{ij}$ is imposed for the existence of at least two positive solutions.

In order to achieve these results, new ideas and techniques have been explored.To begin, we define
\begin{equation*}
\Lambda _{0}:=\frac{3\sqrt{3}\left( p-1\right) \pi \lambda ^{\frac{3}{2}}}{%
32(3-p)A\left( p\right) }\left( \frac{3-p}{2S_{p+1}^{p+1}}\right)
^{2/(p-1)}>0\text{ for }1<p<3,
\end{equation*}%
where
\begin{equation*}
A\left( p\right) =\left\{
\begin{array}{ll}
\left( \frac{3-p}{2}\right) ^{1/\left( p-1\right) }, & \text{ if }1<p\leq 2,
\\
\frac{1}{2}, & \text{ if }2<p<3.%
\end{array}%
\right.
\end{equation*}%
For simplicity, we have assumed $\mu _{11}\leq \mu _{22}$ to facilitate further conditions being imposed on the smaller of the two below. However, the role of $\mu_{11}$ and $\mu_{22}$ can be interchanged while the results remain unchanged.
Then we have the following theorems.

\begin{theorem}
\label{t1}Suppose that $1<p<3$ and $\lambda ,\mu_{ij}>0$. If $0<\mu
_{11}<\Lambda _{0},$ then System $(E)$ has a positive solution $\left(
u_{0},v_{0}\right) $ with positive energy and%
\begin{equation*}
\left\Vert \left( u_{0},v_{0}\right) \right\Vert _{H}\rightarrow 0\text{ as }%
\mu _{22}\rightarrow \infty .
\end{equation*}%
\label{tm1.2}
\end{theorem}
In the proof of Theorem \ref{tm1.2}, we will find critical points by introducing a novel constraint, applying the fibering method while adoptting new analytical techniques.

The next theorem describes our results on the existence of a ground state solution.
\begin{theorem}
\label{t2}Suppose that $2\leq p<3$ and $\lambda ,\mu_{ij}>0$.
Let $\left( u_{0},v_{0}\right) $ be positive solution of System $%
(E)$ as in Theorem \ref{t1}. Then we have\newline
$\left( i\right) $ if $2<p<3,0<\mu _{ii}<\Lambda _{0}$
and $\mu _{11}\mu _{22}-\mu _{12}^{2}\geq 0,$ then $\left(
u_{0},v_{0}\right) $ is a positive ground state solution of System $(E)$;%
\newline
$\left( ii\right) $ if $p=2$ and $0<\mu _{ii}<\Lambda _{0}$ then $\left(
u_{0},v_{0}\right) $ is a positive ground state solution of System $(E).$ %
\label{tm1.3}
\end{theorem}
Note that in setting out the argument for the proof Theorem \ref{tm1.3}, the integral equations (of Nehari and Pohozaev identities) and the required conditions are conveniently written as a linear system of equations with non-linear constraints. This formulation allows us to apply straightforward linear algebra techniques for the otherwise complicated analysis.

The next two theorems cover the case $\mu _{11}\mu _{22}-\mu _{12}^{2}>0$.
We will see that, unlike the case $\mu _{11}\mu _{22}-\mu _{12}^{2}<0$, the
System $(E)$ does not admit any nontrivial solution when $\det \left( \mu
_{ij}\right) =\mu _{11}\mu _{22}-\mu _{12}^{2}$ satisfies suitable
conditions.

\begin{theorem}
\label{T6}Suppose that $1<p\leq 2$ and $\lambda ,\mu _{ij}>0.$ If
\begin{equation*}
\frac{\mu _{11}\mu _{22}-\mu _{12}^{2}}{\mu _{11}+\mu _{22}}>\left\{
\begin{array}{ll}
\frac{\left( p-1\right) ^{2}}{4}\left[ \frac{2^{p}\left( 2-p\right) ^{2-p}}{%
\lambda ^{2-p}}\right] ^{2/\left( p-1\right) }, & \text{ if }1<p<2, \\
4, & \text{ if }p=2,%
\end{array}%
\right.
\end{equation*}%
then System $(E)$ has only trivial solution.
\end{theorem}

Finally, we give the results on the existence of multiple positive solutions.
\begin{theorem}
\label{t3}Suppose that $1<p<2$ and $\lambda ,\mu _{ij}>0.$ If $0<\mu
_{11}<\Lambda _{0}$ and $\mu _{11}\mu _{22}-\mu _{12}^{2}>0,$ then System $%
(E)$ has at least two different positive solutions. \label{tm1.4}
\end{theorem}
The key point in the proof of Theorem \ref{tm1.4}  lies in establishing
Lions type inequalities within the context of the vector functions
(or see \cite{Jin, R1}). Using these inequalities in conjunction with Strauss's inequality in $H_{r}:=H^{1}_{r}\left(\mathbb{R}^{3}\right)\times H^{1}_{r}\left(\mathbb{R}^{3}\right)$ and the comparison of energy, we are able to demonstrate the existence of two different positive solutions.

The rest of this paper is organized as follows. After introducing some
preliminary results in Section 2, we prove Theorem \ref{tm1.2} in Section 3.
In Section 4, the proof of Theorem \ref{tm1.3} is given and, in Section 5, the proof of Theorem \ref{tm1.4}.

\section{Preliminaries}
We first establish the following estimates on the nonlinearity.

\begin{lemma}
\label{L2-0}Suppose that $1<p<2,$ $\lambda ,d>0$ is given. Let $f_{d}\left(
s\right) =\lambda -2^{p}s^{p-1}+ds$ for $s>0.$ Then there exist $d_{\lambda
}:=\left( p-1\right) \left[ \frac{2^{p}\left( 2-p\right) ^{2-p}}{\lambda
^{2-p}}\right] ^{1/\left( p-1\right) }>0$ and $s_{0}\left( d\right) :=\left(
\frac{2^{p}\left( p-1\right) }{d}\right) ^{1/\left( 2-p\right) }>0$
satisfying\newline
$\left( i\right) $ $f_{d}^{\prime }\left( s_{0}\left( d\right) \right) =0$
and $f_{d}\left( s_{0}\left( d_{\lambda }\right) \right) =0;$\newline
$\left( ii\right) $ for each $d<d_{\lambda }$ there exist $\eta _{d},\xi
_{d}>0$ such that $\eta _{d}<s_{0}\left( d\right) <\xi _{d}$ and $%
f_{d}\left( s\right) <0$ for all $s\in \left( \eta _{d},\xi _{d}\right) ;$%
\newline
$\left( iii\right) $ for each $d>d_{\lambda },$ $f_{d}\left( s\right) >0$
for all $s>0.$
\end{lemma}

\begin{proof}
By a straightforward calculation, we can show that the results hold.
\end{proof}

We need the following results.

\begin{lemma}
\label{L2-1}Suppose that $1<p<3$ and $\mu _{ij}>0.$ Let $g\left( s\right)
=\mu _{11}s^{2}+\mu _{22}\left( 1-s\right) ^{2}-2\mu _{12}s\left( 1-s\right)
$ for $s\in \left[ 0,1\right] .$ Then there exists $0<s_{\min }=\frac{\mu
_{22}+\mu _{12}}{\mu _{11}+\mu _{22}+2\mu _{12}}<1$ such that $\min_{s\in %
\left[ 0,1\right] }g\left( s\right) =g\left( s_{\min }\right) =\frac{\mu
_{11}\mu _{22}-\mu _{12}^{2}}{\mu _{11}+\mu _{22}+2\mu _{12}}<\mu _{ii}$ for
$i=1,2.$
\end{lemma}

\begin{proof}
Since
\begin{eqnarray*}
g\left( s\right) &=&\mu _{11}s^{2}+\mu _{22}\left( 1-s\right) ^{2}-2\mu
_{12}s\left( 1-s\right) \\
&=&\left( \mu _{11}+\mu _{22}+2\mu _{12}\right) s^{2}-2\left( \mu _{22}+\mu
_{12}\right) s+\mu _{22}
\end{eqnarray*}%
and%
\begin{equation*}
g^{\prime }\left( s\right) =2\left( \mu _{11}+\mu _{22}+2\mu _{12}\right)
s-2\left( \mu _{22}+\mu _{12}\right) ,
\end{equation*}%
we conclude that there exists%
\begin{equation*}
0<s_{\min }=\frac{\mu _{22}+\mu _{12}}{\mu _{11}+\mu _{22}+2\mu _{12}}<1
\end{equation*}%
such that%
\begin{eqnarray*}
\min_{s\in \left[ 0,1\right] }g\left( s\right) &=&g\left( s_{\min }\right)
=g\left( \frac{\mu _{22}+\mu _{12}}{\mu _{11}+\mu _{22}+2\mu _{12}}\right) =%
\frac{\left( \mu _{22}+\mu _{12}\right) ^{2}}{\mu _{11}+\mu _{22}+2\mu _{12}}%
-2\frac{\left( \mu _{22}+\mu _{12}\right) ^{2}}{\mu _{11}+\mu _{22}+2\mu
_{12}}+\mu _{22} \\
&=&\mu _{22}-\frac{\left( \mu _{22}+\mu _{12}\right) ^{2}}{\mu _{11}+\mu
_{22}+2\mu _{12}} \\
&=&\frac{\mu _{11}\mu _{22}-\mu _{12}^{2}}{\mu _{11}+\mu _{22}+2\mu _{12}}%
<\mu _{ii}\text{ for }i=1,2.
\end{eqnarray*}%
This completes the proof.
\end{proof}

The function $\phi _{u}$ defined in $\left( \ref{1-2}\right) $ possesses
certain properties \cite{AP,R1} and the Hardy-Littlewood-Sobolev and
Gagliardo-Nirenberg inequalities, thus we have the following results.

\begin{lemma}
\label{L2-3}For each $u\in H^{1}\left( \mathbb{R}^{3}\right) $, the
following two inequalities are true.
\end{lemma}

\begin{itemize}
\item[$\left( i\right) $] $\phi _{u}\geq 0;$

\item[$\left( ii\right) $] $\int_{\mathbb{R}^{3}}\phi _{u}u^{2}dx\leq \frac{%
16}{3\sqrt{3}\pi \lambda ^{\frac{3}{2}}}\left( \int_{\mathbb{R}^{3}}\lambda
u^{2}dx\right) ^{\frac{3}{2}}\left( \int_{\mathbb{R}^{3}}|\nabla
u|^{2}dx\right) ^{\frac{1}{2}}$ for $\lambda >0.$
\end{itemize}

Next, we consider the following Schr\"{o}dinger--Poisson equation:
\begin{equation}
\begin{array}{ll}
-\Delta u+\lambda u+\mu \int_{\mathbb{R}^{3}}\phi _{_{u}}u^{2}=\left\vert
u\right\vert ^{p-2}u & \text{ in }\mathbb{R}^{3}.%
\end{array}
\tag*{$\left( SP_{\mu }\right) $}
\end{equation}%
Such equation is variational and its solutions are critical points of the
corresponding energy functional $I:H^{1}\left( \mathbb{R}^{3}\right)
\rightarrow \mathbb{R}$ defined as%
\begin{equation*}
I_{\mu }(u)=\frac{1}{2}\int_{\mathbb{R}^{3}}|\nabla u|^{2}+\lambda u^{2}dx+%
\frac{\mu }{4}\int_{\mathbb{R}^{3}}\phi _{_{u}}u^{2}dx-\frac{1}{p+1}\int_{%
\mathbb{R}^{3}}\left\vert u\right\vert ^{p+1}dx.
\end{equation*}%
Note that $u\in H^{1}\left( \mathbb{R}^{3}\right) $ is a solution of
Equation $\left( SP_{\mu }\right) $ if and only if $u$ is a critical point
of $I_{\mu }.$ Next, we define the Nehari manifold of functional $I_{\mu }$
as follows,
\begin{equation*}
\mathbf{N}_{\mu }:=\{u\in H^{1}\left( \mathbb{R}^{3}\right) \backslash
\{0\}:\left\langle I_{\mu }^{\prime }\left( u\right) ,u\right\rangle =0\}.
\end{equation*}%
The Nehari manifold $\mathbf{N}_{\mu }$ is closely linked to the behavior of
the function of the form $f_{u}:t\rightarrow I_{\mu }\left( tu\right) $ for $%
t>0.$ Such maps are known as fibering maps and were introduced by Dr\'{a}%
bek-Pohozaev \cite{DP}, and were further discussed by Brown-Zhang \cite{BZ},
Brown-Wu \cite{BW1,BW2} and many others. For $u\in H^{1}\left( \mathbb{R}^{3}\right)
,$ we find%
\begin{eqnarray*}
f_{u}\left( t\right) &=&\frac{t^{2}}{2}\int_{\mathbb{R}^{3}}|\nabla
u|^{2}+\lambda u^{2}dx+\frac{t^{4}\mu }{4}\int_{\mathbb{R}^{3}}\phi
_{_{u}}u^{2}dx-\frac{t^{p+1}}{p+1}\int_{\mathbb{R}^{3}}\left\vert
u\right\vert ^{p+1}dx, \\
f_{u}^{\prime }\left( t\right) &=&t\int_{\mathbb{R}^{3}}|\nabla
u|^{2}+\lambda u^{2}dx+t^{3}\mu \int_{\mathbb{R}^{3}}\phi
_{_{u}}u^{2}dx-t^{p}\int_{\mathbb{R}^{3}}\left\vert u\right\vert ^{p+1}dx, \\
f_{u}^{\prime \prime }\left( t\right) &=&\int_{\mathbb{R}^{3}}|\nabla
u|^{2}+\lambda u^{2}dx+3t^{2}\mu \int_{\mathbb{R}^{3}}\phi
_{_{u}}u^{2}dx-pt^{p-1}\int_{\mathbb{R}^{3}}\left\vert u\right\vert ^{p+1}dx.
\end{eqnarray*}%
As a direct consequence, we have
\begin{equation*}
tf_{u}^{\prime }\left( t\right) =\int_{\mathbb{R}^{3}}|\nabla
tu|^{2}+\lambda \left( tu\right) ^{2}dx+\mu \int_{\mathbb{R}^{3}}\phi
_{_{tu}}\left( tu\right) ^{2}dx-\int_{\mathbb{R}^{3}}\left\vert
tu\right\vert ^{p+1}dx,
\end{equation*}%
and so, for $u\in H^{1}\left( \mathbb{R}^{3}\right) $ and $t>0,$ $%
f_{u}^{\prime }\left( t\right) =0$ holds if and only if $tu\in \mathbf{N}%
_{\mu }.$ In particular, $f_{u}^{\prime }\left( 1\right) =0$  if and
only if $u\in \mathbf{N}_{\mu }.$ It is then natural to split $\mathbf{N}%
_{\mu }$ into three parts corresponding to the local minima, local maxima
and points of inflection. Following \cite{T}, we define
\begin{eqnarray*}
\mathbf{N}_{\mu }^{+} &=&\{u\in \mathbf{N}_{\mu }:f_{u}^{\prime \prime
}\left( 1\right) >0\}, \\
\mathbf{N}_{\mu }^{0} &=&\{u\in \mathbf{N}_{\mu }:f_{u}^{\prime \prime
}\left( 1\right) =0\}, \\
\mathbf{N}_{\mu }^{-} &=&\{u\in \mathbf{N}_{\mu }:f_{u}^{\prime \prime
}\left( 1\right) <0\}.
\end{eqnarray*}%
Let
\begin{equation*}
\beta _{\mu }:=\inf_{u\in \mathbf{N}_{\mu }^{-}}I_{\mu }\left( u\right) .
\end{equation*}%
Using the argument of theorem 1.3 and lemma 2.4 in \cite{SWF} (or see
Lemma \ref{g7} in below), for each $1<p<3$ and $0<\mu <\Lambda _{0},$
Equation $\left( SP_{\mu }\right) $ has a positive solution $w_{\mu }\in
\mathbf{N}_{\mu }^{-}$ such that%
\begin{equation*}
\left\Vert w_{\mu }\right\Vert _{H^{1}}<\left( \frac{3\sqrt{3}\left(
p-1\right) \pi \lambda ^{\frac{3}{2}}}{16\mu (3-p)}\right) ^{1/2}
\end{equation*}%
and%
\begin{equation*}
d_{0}<\beta _{\mu }=I_{\mu }\left( w_{\mu }\right) <\frac{A\left( p\right)
\left( p-1\right) }{2\left( p+1\right) }\left( \frac{2S_{p+1}^{p+1}}{3-p}%
\right) ^{2/(p-1)}\text{ for some }d_{0}>0.
\end{equation*}%
In particular, by \cite{CKW}, if $2\leq p<3,$ then $w_{\mu }$ is a positive
ground state solution of Equation $\left( SP_{\mu }\right) .$ Moreover, by
\begin{equation*}
I_{\mu _{1}}(u)\leq I_{\mu _{2}}(u)\text{ for all }\mu _{1}<\mu _{2}
\end{equation*}%
and
\begin{equation*}
\beta _{\mu }:=\inf_{u\in \mathbf{N}_{\mu }^{-}}I_{\mu }\left( u\right)
=+\infty ,\text{ if }\mathbf{N}_{\mu }^{-}=\emptyset ,
\end{equation*}%
we may assume that $\beta _{\mu _{1}}<\beta _{\mu _{2}}$ for all $\mu
_{1}<\mu _{2}.$ Thus, by \cite{SWF, R1}, if $1<p<2,$ we have the
following result.

\begin{theorem}
\label{T5}Suppose that $1<p<2.$ Then for each $0<\mu <\Lambda _{0},$
Equation $\left( SP_{\mu }\right) $ has at least two positive radial
solutions $w_{r,\mu }^{\left( 1\right) }$ and $w_{r,\mu }^{\left( 2\right) }$
with%
\begin{equation*}
I_{\mu }\left( w_{r,\mu }^{\left( 1\right) }\right) =\beta _{r,\mu }^{\left(
1\right) }:=\inf_{u\in \mathbf{N}_{\mu }^{-}\cap H_{r}^{1}\left( \mathbb{R}%
^{3}\right) }I_{\mu }\left( u\right) >0
\end{equation*}%
and
\begin{equation*}
I_{\mu }\left( w_{r,\mu }^{\left( 2\right) }\right) =\beta _{r,\mu }^{\left(
2\right) }:=\inf_{u\in \mathbf{N}_{\mu }^{+}\cap H_{r}^{1}\left( \mathbb{R}%
^{3}\right) }I_{\mu }\left( u\right) =\inf_{u\in H_{r}^{1}\left( \mathbb{R}%
^{3}\right) }I_{\mu }\left( u\right) <0.
\end{equation*}
\end{theorem}

\section{Positive vectorial solutions}

First, we define the Nehari manifold of functional $J$ as follows.
\begin{equation*}
\mathbf{M}:=\{\left( u,v\right) \in H\backslash \{\left( 0,0\right)
\}:F\left( u,v\right) :=\left\langle J^{\prime }\left( u,v\right) ,\left(
u,v\right) \right\rangle =0\},
\end{equation*}%
where%
\begin{eqnarray*}
F\left( u,v\right) &=&\left\Vert \left( u,v\right) \right\Vert
_{H}^{2}+\int_{\mathbb{R}^{3}}\mu _{11}\phi _{_{u}}u^{2}+\mu _{22}\phi
_{_{v}}v^{2}-2\mu _{12}\phi _{_{v}}u^{2}dx \\
&&-\frac{1}{2\pi }\int_{\mathbb{R}^{3}}\int_{0}^{2\pi }\left\vert
u+e^{i\theta }v\right\vert ^{p+1}d\theta dx\\
&=&\int_{\mathbb{R}^{3}}|\nabla u|^{2}+\lambda u^{2}dx+\frac{1}{2}%
\int_{\mathbb{R}^{3}}|\nabla v|^{2}+\lambda v^{2}dx \\
&&+\int_{\mathbb{R}^{3}}\mu _{11}\phi _{_{u}}u^{2}+\mu _{22}\phi
_{_{v}}v^{2}-2\mu _{12}\phi _{_{v}}u^{2}dx \\
&&-\frac{1}{2\pi }\int_{\mathbb{R}^{3}}\int_{0}^{2\pi
}\left( u^{2}+2uv\cos \theta +v^{2}\right) ^{\frac{p+1}{2}}d\theta dx.
\end{eqnarray*}%
Then $u\in \mathbf{M}$ if and only if $\left\langle J^{\prime }\left(
u,v\right) ,\left( u,v\right) \right\rangle =0.$ It follows the Sobolev and
Young inequalities that%
\begin{eqnarray*}
\left\Vert \left( u,v\right) \right\Vert _{H}^{2}-2\mu _{12}\widehat{C}%
_{0}\left\Vert \left( u,v\right) \right\Vert _{H}^{4} &\leq &\left\Vert
\left( u,v\right) \right\Vert _{H}^{2}+\int_{\mathbb{R}^{3}}\mu _{11}\phi
_{_{u}}u^{2}+\mu _{22}\phi _{_{v}}v^{2}-2\mu _{12}\phi _{_{v}}u^{2}dx \\
&=&\frac{1}{2\pi }\int_{\mathbb{R}^{3}}\int_{0}^{2\pi }\left\vert
u+e^{i\theta }v\right\vert ^{p+1}d\theta dx \\
&\leq &C_{0}\left\Vert \left( u,v\right) \right\Vert _{H}^{p+1}\text{ for
all }u\in \mathbf{M}.
\end{eqnarray*}%
Since $1<p<3,$ there exists $C_{\mu _{12}}>0$ with $C_{\mu _{12}}\rightarrow
0$ as $\mu _{12}\rightarrow \infty $ such that%
\begin{equation}
\left\Vert \left( u,v\right) \right\Vert _{H}\geq C_{\mu _{12}}\text{ for
all }u\in \mathbf{M}.  \label{2-2}
\end{equation}

The Nehari manifold $\mathbf{M}$ is closely linked to the behavior of the
function of the form $h_{\left( u,v\right) }:t\rightarrow J\left(
tu,tv\right) $ for $t>0.$ For $\left( u,v\right) \in H,$ we find%
\begin{eqnarray*}
h_{\left( u,v\right) }\left( t\right) &=&\frac{t^{2}}{2}\left\Vert \left(
u,v\right) \right\Vert _{H}^{2}+\frac{t^{4}}{4}\int_{\mathbb{R}^{3}}\mu
_{11}\phi _{_{u}}u^{2}+\mu _{22}\phi _{_{v}}v^{2}-2\mu _{12}\phi
_{_{v}}u^{2}dx \\
&&-\frac{t^{p+1}}{2\pi \left( p+1\right) }\int_{\mathbb{R}%
^{3}}\int_{0}^{2\pi }\left\vert u+e^{i\theta }v\right\vert ^{p+1}d\theta dx,
\\
h_{\left( u,v\right) }^{\prime }\left( t\right) &=&t\left\Vert \left(
u,v\right) \right\Vert _{H}^{2}+t^{3}\int_{\mathbb{R}^{3}}\mu _{11}\phi
_{_{u}}u^{2}+\mu _{22}\phi _{_{v}}v^{2}-2\mu _{12}\phi _{_{v}}u^{2}dx \\
&&-\frac{t^{p}}{2\pi }\int_{\mathbb{R}^{3}}\int_{0}^{2\pi }\left\vert
u+e^{i\theta }v\right\vert ^{p+1}d\theta dx, \\
h_{\left( u,v\right) }^{\prime \prime }\left( t\right) &=&\left\Vert \left(
u,v\right) \right\Vert _{H}^{2}+3t^{2}\int_{\mathbb{R}^{3}}\mu _{11}\phi
_{_{u}}u^{2}+\mu _{22}\phi _{_{v}}v^{2}-2\mu _{12}\phi _{_{v}}u^{2}dx \\
&&-\frac{pt^{p-1}}{2\pi }\int_{\mathbb{R}^{3}}\int_{0}^{2\pi }\left\vert
u+e^{i\theta }v\right\vert ^{p+1}d\theta dx.
\end{eqnarray*}%
As a direct consequence, we have
\begin{eqnarray*}
th_{\left( u,v\right) }^{\prime }\left( t\right) &=&\left\Vert \left(
tu,tv\right) \right\Vert _{H}^{2}+\int_{\mathbb{R}^{3}}\mu _{11}\phi
_{_{tu}}t^{2}u^{2}+\mu _{22}\phi _{_{tv}}t^{2}v^{2}-2\mu _{12}\phi
_{_{tv}}t^{2}u^{2}dx \\
&&-\frac{1}{2\pi }\int_{\mathbb{R}^{3}}\int_{0}^{2\pi }\left\vert
tu+e^{i\theta }tv\right\vert ^{p+1}d\theta dx,
\end{eqnarray*}%
and so, for $\left( u,v\right) \in H\backslash \left\{ \left( 0,0\right)
\right\} $ and $t>0,$ $h_{\left( u,v\right) }^{\prime }\left( t\right) =0$
holds if and only if $\left( tu,tv\right) \in \mathbf{M}$. In particular, $%
h_{\left( u,v\right) }^{\prime }\left( 1\right) =0$  if and only if $%
\left( u,v\right) \in \mathbf{M}.$ It is then natural to split $\mathbf{M}$
into three parts corresponding to the local minima, local maxima and points
of inflection. Following \cite{T}, we define
\begin{eqnarray*}
\mathbf{M}^{+} &=&\{u\in \mathbf{M}:h_{\left( u,v\right) }^{\prime \prime
}\left( 1\right) >0\}, \\
\mathbf{M}^{0} &=&\{u\in \mathbf{M}:h_{\left( u,v\right) }^{\prime \prime
}\left( 1\right) =0\}, \\
\mathbf{M}^{-} &=&\{u\in \mathbf{M}:h_{\left( u,v\right) }^{\prime \prime
}\left( 1\right) <0\}.
\end{eqnarray*}

\begin{lemma}
\label{g2}Suppose that $\left( u_{0},v_{0}\right) $ is a local minimizer for
$J$ on $\mathbf{M}$ and $\left( u_{0},v_{0}\right) \notin \mathbf{M}^{0}.$
Then $J^{\prime }\left( u_{0},v_{0}\right) =0$ in $H^{-1}.$
\end{lemma}

\begin{proof}
The proof of Lemma \ref{g2} is essentially the same as that in Brown-Zhang \cite[%
Theorem 2.3]{BZ} (or see Binding-Dr\'{a}bek-Huang \cite{BDH}) and is subsequently omitted here.
\end{proof}

For each $\left( u,v\right) \in \mathbf{M},$ we find that
\begin{eqnarray}
h_{\left( u,v\right) }^{\prime \prime }\left( 1\right) &=&\left\Vert \left(
u,v\right) \right\Vert _{H}^{2}+3\int_{\mathbb{R}^{3}}\mu _{11}\phi
_{_{u}}u^{2}+\mu _{22}\phi _{_{v}}v^{2}-2\mu _{12}\phi _{_{v}}u^{2}dx-\frac{p%
}{2\pi }\int_{\mathbb{R}^{3}}\int_{0}^{2\pi }\left\vert tu+e^{i\theta
}tv\right\vert ^{p+1}d\theta dx  \notag \\
&=&-\left( p-1\right) \left\Vert \left( u,v\right) \right\Vert
_{H}^{2}+\left( 3-p\right) \int_{\mathbb{R}^{3}}\mu _{11}\phi
_{_{u}}u^{2}+\mu _{22}\phi _{_{v}}v^{2}-2\mu _{12}\phi _{_{v}}u^{2}dx
\label{2-6-1} \\
&=&-2\left\Vert \left( u,v\right) \right\Vert _{H}^{2}+\frac{3-p}{2\pi }%
\int_{\mathbb{R}^{3}}\int_{0}^{2\pi }\left\vert tu+e^{i\theta }tv\right\vert
^{p+1}d\theta dx.  \label{2-6-2}
\end{eqnarray}%
For each $\left( u,v\right) \in \mathbf{M}^{-}$, using $\left( \ref{2-2}%
\right) $ and $(\ref{2-6-2})$ gives
\begin{eqnarray*}
J(u,v) &=&\frac{1}{4}\left\Vert \left( u,v\right) \right\Vert _{H}^{2}-\frac{%
3-p}{8\pi \left( p+1\right) }\int_{\mathbb{R}^{3}}\int_{0}^{2\pi }\left\vert
tu+e^{i\theta }tv\right\vert ^{p+1}d\theta dx \\
&>&\frac{p-1}{4\left( p+1\right) }\left\Vert \left( u,v\right) \right\Vert
_{H}^{2}\geq \frac{p-1}{4\left( p+1\right) }C_{\mu _{12}}^{2}>0,
\end{eqnarray*}%
and for each $\left( u,v\right) \in \mathbf{M}^{+},$%
\begin{eqnarray*}
J(u,v) &=&\frac{p-1}{2\left( p+1\right) }\left\Vert \left( u,v\right)
\right\Vert _{H}^{2}-\left( \frac{3-p}{4\left( p+1\right) }\right) \int_{%
\mathbb{R}^{3}}\mu _{11}\phi _{_{u}}u^{2}+\mu _{22}\phi _{_{v}}v^{2}-2\mu
_{12}\phi _{_{v}}u^{2}dx \\
&<&\frac{p-2}{4p}\left\Vert \left( u,v\right) \right\Vert _{H}^{2}.
\end{eqnarray*}%
Hence, we obtain the following result.

\begin{lemma}
\label{g5}The energy functional $J$ is coercive and bounded below on $%
\mathbf{M}^{-}.$ Furthermore, for all $u\in \mathbf{M}^{-}$,
\begin{equation*}
J(u,v)>\frac{p-1}{4\left( p+1\right) }C_{\mu _{12}}^{2}>0.
\end{equation*}
\end{lemma}

For $0<\mu _{11}<\Lambda _{0},$ let $\left( u,v\right) \in \mathbf{M}$ with $%
J\left( u,v\right) <\frac{A\left( p\right) \left( p-1\right) }{2\left(
p+1\right) }\left( \frac{2S_{p+1}^{p+1}}{3-p}\right) ^{2/(p-1)}.$ Since $\mu
_{11}\leq \mu _{22},$ by Lemma \ref{L2-3}, we deduce that%
\begin{eqnarray*}
\frac{A\left( p\right) \left( p-1\right) }{2\left( p+1\right) }\left( \frac{%
2S_{p+1}^{p+1}}{3-p}\right) ^{2/(p-1)} &>&J(u,v) \\
&=&\frac{p-1}{2\left( p+1\right) }\left\Vert \left( u,v\right) \right\Vert
_{H}^{2}-\frac{3-p}{4\left( p+1\right) }\int_{\mathbb{R}^{3}}\mu _{11}\phi
_{_{u}}u^{2}+\mu _{22}\phi _{_{v}}v^{2}-2\mu _{12}\phi _{_{v}}u^{2}dx \\
&\geq &\frac{p-1}{2\left( p+1\right) }\left\Vert \left( u,v\right)
\right\Vert _{H}^{2}-\frac{3-p}{4\left( p+1\right) }\int_{\mathbb{R}^{3}}\mu
_{11}\phi _{_{u}}u^{2}+\mu _{22}\phi _{_{v}}v^{2}dx \\
&\geq &\frac{p-1}{2\left( p+1\right) }\left\Vert \left( u,v\right)
\right\Vert _{H}^{2}-\frac{\mu _{22}(3-p)}{4\left( p+1\right) }\frac{16}{3%
\sqrt{3}\pi \lambda ^{\frac{3}{2}}}\left\Vert \left( u,v\right) \right\Vert
_{H}^{4} \\
&=&\frac{p-1}{2\left( p+1\right) }\left\Vert \left( u,v\right) \right\Vert
_{H}^{2}-\frac{4\mu _{22}(3-p)}{3\sqrt{3}\left( p+1\right) \pi \lambda ^{%
\frac{3}{2}}}\left\Vert \left( u,v\right) \right\Vert _{H}^{4}.
\end{eqnarray*}%
Since the function
\begin{equation*}
q\left( x\right) =\frac{p-1}{2\left( p+1\right) }x^{2}-\frac{4\mu _{22}(3-p)%
}{3\sqrt{3}\left( p+1\right) \pi \lambda ^{\frac{3}{2}}}x^{4}
\end{equation*}%
has the maximum at $x_{0}=\left( \frac{3\sqrt{3}\left( p-1\right) \pi
\lambda ^{\frac{3}{2}}}{16\mu _{22}(3-p)}\right) ^{1/2},$ we have%
\begin{equation*}
\max_{x\geq 0}q\left( x\right) =q\left( x_{0}\right) =\frac{3\sqrt{3}\left(
p-1\right) ^{2}\pi \lambda ^{\frac{3}{2}}}{64\mu _{22}(3-p)\left( p+1\right)
}\geq \frac{A\left( p\right) \left( p-1\right) }{2\left( p+1\right) }\left(
\frac{2S_{p+1}^{p+1}}{3-p}\right) ^{2/(p-1)}.
\end{equation*}%
Thus,
\begin{equation}
\mathbf{M}\left[ \frac{A\left( p\right) \left( p-1\right) }{2\left(
p+1\right) }\left( \frac{2S_{p+1}^{p+1}}{3-p}\right) ^{2/(p-1)}\right] =%
\mathbf{M}^{(1)}\cup \mathbf{M}^{(2)},  \label{4-4}
\end{equation}%
where%
\begin{equation*}
\mathbf{M}\left[ \frac{A\left( p\right) \left( p-1\right) }{2\left(
p+1\right) }\left( \frac{2S_{p+1}^{p+1}}{3-p}\right) ^{2/(p-1)}\right]
:=\left\{ u\in \mathbf{M}:J\left( u,v\right) <\frac{A\left( p\right) \left(
p-1\right) }{2\left( p+1\right) }\left( \frac{2S_{p+1}^{p+1}}{3-p}\right)
^{2/(p-1)}\right\} ,
\end{equation*}%
\begin{equation*}
\mathbf{M}^{(1)}:=\left\{ u\in \mathbf{M}\left[ \frac{A\left( p\right)
\left( p-1\right) }{2\left( p+1\right) }\left( \frac{2S_{p+1}^{p+1}}{3-p}%
\right) ^{2/(p-1)}\right] :\left\Vert \left( u,v\right) \right\Vert
_{H}<\left( \frac{3\sqrt{3}\left( p-1\right) \pi \lambda ^{\frac{3}{2}}}{%
16\mu _{22}(3-p)}\right) ^{1/2}\right\},
\end{equation*}%
and
\begin{equation*}
\mathbf{M}^{(2)}:=\left\{ u\in \mathbf{M}\left[ \frac{A\left( p\right)
\left( p-1\right) }{2\left( p+1\right) }\left( \frac{2S_{p+1}^{p+1}}{3-p}%
\right) ^{2/(p-1)}\right] :\left\Vert \left( u,v\right) \right\Vert
_{H}>\left( \frac{3\sqrt{3}\left( p-1\right) \pi \lambda ^{\frac{3}{2}}}{%
16\mu _{22}(3-p)}\right) ^{1/2}\right\} .
\end{equation*}%
By $\left( \ref{2-6-1}\right) $ and Lemma \ref{L2-3}, it follows from the Sobolev
inequality that
\begin{eqnarray*}
h_{\left( u,v\right) }^{\prime \prime }\left( 1\right) &=&-\left( p-1\right)
\left\Vert \left( u,v\right) \right\Vert _{H}^{2}+\left( 3-p\right) \int_{%
\mathbb{R}^{3}}\mu _{11}\phi _{_{u}}u^{2}+\mu _{22}\phi _{_{v}}v^{2}-2\mu
_{12}\phi _{_{v}}u^{2}dx \\
&\leq &\left\Vert \left( u,v\right) \right\Vert _{H}^{2}\left[ \frac{16\mu
_{22}(3-p)}{3\sqrt{3}\pi \lambda ^{\frac{3}{2}}}\left\Vert \left( u,v\right)
\right\Vert _{H}^{2}-\left( p-1\right) \right] \\
&<&\left\Vert \left( u,v\right) \right\Vert _{H}^{2}\left( \frac{16\mu
_{22}(3-p)}{3\sqrt{3}\pi \lambda ^{\frac{3}{2}}}\frac{3\sqrt{3}\left(
p-1\right) \pi \lambda ^{\frac{3}{2}}}{16\mu _{22}(3-p)}-\left( p-1\right)
\right) \\
&=&0\text{ for all }u\in \mathbf{M}^{(1)}.
\end{eqnarray*}%
Using $\left( \ref{2-6-2}\right),$ we deduce that
\begin{eqnarray*}
&&\frac{1}{4}\left\Vert \left( u,v\right) \right\Vert _{H}^{2}-\frac{3-p}{%
8\pi \left( p+1\right) }\int_{\mathbb{R}^{3}}\int_{0}^{2\pi }\left\vert
tu+e^{i\theta }tv\right\vert ^{p+1}d\theta dx \\
&=&J\left( u,v\right) <\frac{3\sqrt{3}\left( p-1\right) ^{2}\pi \lambda ^{%
\frac{3}{2}}}{64\mu _{22}(3-p)\left( p+1\right) } \\
&<&\frac{p-1}{4\left( p+1\right) }\left\Vert \left( u,v\right) \right\Vert
_{H}^{2}\text{ for all }u\in \mathbf{M}^{(2)},
\end{eqnarray*}%
which implies that if $u\in \mathbf{M}^{(2)},$ then%
\begin{equation*}
h_{\left( u,v\right) }^{\prime \prime }\left( 1\right) =-2\left\Vert \left(
u,v\right) \right\Vert _{H}^{2}+\frac{3-p}{2\pi }\int_{\mathbb{R}%
^{3}}\int_{0}^{2\pi }\left\vert tu+e^{i\theta }tv\right\vert ^{p+1}d\theta
d>0.
\end{equation*}%
Hence, we obtain the following results.

\begin{lemma}
\label{g7}Suppose that $1<p<3$ and $\mu _{ij}>0.$ If $0<\mu _{11}<\Lambda
_{0},$ Then $\mathbf{M}^{(1)}\subset \mathbf{M}^{-}$ and $\mathbf{M}%
^{(2)}\subset \mathbf{M}^{+}$ are $C^{1}$ sub-manifolds. Furthermore, each
local minimizer of the functional $J$ in the sub-manifolds $\mathbf{M}^{(1)}$
and $\mathbf{M}^{(2)}$ is a critical point of $J$ in $H.$
\end{lemma}

We have the following results.

\begin{lemma}
\label{l5}Suppose that $1<p<3$ and $0<\mu _{11}\leq \mu _{22}.$ Let $\left(
u_{0},v_{0}\right) $ be a critical point of $J$ on $\mathbf{M}^{-}.$ Then we
have $J\left( u_{0},v_{0}\right) \geq \beta _{\mu _{11}}$ if either $u_{0}=0$
or $v_{0}=0.$
\end{lemma}

\begin{proof}
Since $\beta _{\mu _{11}}\leq \beta _{\mu _{22}}$ for $\mu _{11}\leq \mu
_{22},$ without loss of generality, we may assume that $v_{0}=0.$ Then
\begin{equation*}
J\left( u_{0},0\right) =I_{\mu _{11}}\left( u_{0}\right) =\frac{1}{2}%
\left\Vert u_{0}\right\Vert _{H^{1}}^{2}+\frac{\mu _{11}}{4}\int_{\mathbb{R}%
^{3}}\phi _{u_{0}}u_{0}^{2}dx-\frac{1}{p+1}\int_{\mathbb{R}^{3}}\left\vert
u_{0}\right\vert ^{p+1}dx
\end{equation*}%
and%
\begin{equation*}
h_{\left( u,0\right) }^{\prime \prime }\left( 1\right) =f_{u}^{\prime \prime
}\left( 1\right) =-2\left\Vert u_{0}\right\Vert _{H^{1}}^{2}+\left(
3-p\right) \int_{\mathbb{R}^{3}}\left\vert u_{0}\right\vert ^{p+1}dx<0,
\end{equation*}%
which implies that $u_{0}\in \mathbf{N}_{\mu _{11}}^{-}.$ Thus $J\left(
u_{0},0\right) =I_{\mu _{11}}\left( u_{0}\right) \geq \beta _{\mu _{11}}.$
Consequently, we complete the proof.
\end{proof}

\begin{lemma}
\label{g4}Suppose that $1<p<3$ and $\mu _{ij}>0.$ Let $0<\mu _{11}<\Lambda
_{0}$ and let $w_{\mu _{11}}$ be a positive solution of Equation $\left(
SP_{\mu _{11}}\right) $ with $I_{\mu _{11}}\left( w_{\mu _{11}}\right)
=\beta _{\mu _{11}}.$ Then we have the following results.\newline
$\left( i\right) $ If $\det \left( \mu _{ij}\right) >0,$ then
there exist two constants $t_{0}^{+}$ and $t_{0}^{-}$ which satisfy
\begin{equation*}
0<t_{0}^{-}<\left( \frac{2\left\Vert w_{\mu _{11}}\right\Vert _{H^{1}}^{2}}{%
\left( 3-p\right) \int_{\mathbb{R}^{3}}w_{\mu _{11}}^{p+1}dx}\right)
^{1/\left( p-1\right) }<t_{0}^{+},
\end{equation*}%
such that%
\begin{equation*}
\left( t_{0}^{\pm }\sqrt{s_{\min }}w_{\mu _{11}},t_{0}^{\pm }\sqrt{1-s_{0}}%
w_{\mu _{11}}\right) \in \mathbf{M}^{\pm }
\end{equation*}%
and%
\begin{eqnarray*}
J\left( t_{0}^{-}\sqrt{s_{\min }}w_{\mu _{11}},t_{0}^{-}\sqrt{1-s_{\min }}%
w_{\mu _{11}}\right) &<&\beta _{\mu _{11}}, \\
J\left( t_{0}^{+}\sqrt{s_{\min }}w_{\mu _{11}},t_{0}^{+}\sqrt{1-s_{\min }}%
w_{\mu _{11}}\right) &=&\inf_{t\geq 0}J\left( t\sqrt{s_{\min }}w_{\mu
_{11}},t\sqrt{1-s_{\min }}w_{\mu _{11}}\right) <0,
\end{eqnarray*}%
where $s_{\min }=\frac{\mu _{22}+\mu _{12}}{\mu _{11}+\mu _{22}+2\mu _{12}}$
as in Lemma \ref{L2-1}. In particular,%
\begin{equation*}
\left( t_{0}^{-}\sqrt{s_{\min }}w_{\mu _{11}},t_{0}^{-}\sqrt{1-s_{\min }}%
w_{\mu _{11}}\right) \in \mathbf{M}^{\left( 1\right) }
\end{equation*}%
and
\begin{equation*}
\left( t_{0}^{+}\sqrt{s_{\min }}w_{\mu _{11}},t_{0}^{+}\sqrt{1-s_{\min }}%
w_{\mu _{11}}\right) \in \mathbf{M}^{\left( 2\right) }.
\end{equation*}%
$\left( ii\right) $ If $\det \left( \mu _{i,j}\right) \leq 0,$ then there
exists a constant $t_{0}^{-}$ which satisfy
\begin{equation*}
0<t_{0}^{-}<\left( \frac{2\left\Vert w_{\mu _{11}}\right\Vert _{H^{1}}^{2}}{%
\left( 3-p\right) \int_{\mathbb{R}^{3}}w_{\mu _{11}}^{p+1}dx}\right)
^{1/\left( p-1\right) },
\end{equation*}%
such that%
\begin{equation*}
\left( t_{0}^{-}\sqrt{s_{\min }}w_{\mu _{11}},t_{0}^{-}\sqrt{1-s_{\min }}%
w_{\mu _{11}}\right) \in \mathbf{M}^{\left( 1\right) }
\end{equation*}%
and%
\begin{equation*}
J\left( t_{0}^{-}\sqrt{s_{\min }}w_{\mu _{11}},t_{0}^{-}\sqrt{1-s_{\min }}%
w_{\mu _{11}}\right) <\beta _{\mu _{ii}}.
\end{equation*}
\end{lemma}

\begin{proof}
Define $w_{0}\left( x\right) :=w_{\mu _{11}}\left( x\right) $ and
\begin{eqnarray*}
\eta \left( t\right) &=&t^{-2}\left\Vert \left( \sqrt{s_{\min }}w_{0},\sqrt{%
1-s_{\min }}w_{0}\right) \right\Vert _{H}^{2}-\frac{t^{p-3}}{2\pi }\int_{%
\mathbb{R}^{3}}\int_{0}^{2\pi }\left\vert \sqrt{s_{\min }}w_{0}+e^{i\theta }%
\sqrt{1-s_{\min }}w_{0}\right\vert ^{p+1}d\theta dx \\
&=&t^{-2}\left\Vert w_{0}\right\Vert _{H^{1}}^{2}-\frac{t^{p-3}}{2\pi }\int_{%
\mathbb{R}^{3}}\int_{0}^{2\pi }\left( s_{\min }w_{0}^{2}+2\sqrt{s_{\min
}\left( 1-s_{\min }\right) }w_{0}^{2}\cos \theta +\left( 1-s_{\min }\right)
w_{0}^{2}\right) ^{\left( p+1\right) /2}d\theta dx \\
&=&t^{-2}\left\Vert w_{0}\right\Vert _{H^{1}}^{2}-\frac{t^{p-3}}{2\pi }\int_{%
\mathbb{R}^{3}}\int_{0}^{2\pi }\left( 1+2\sqrt{s_{\min }\left( 1-s_{\min
}\right) }\cos \theta \right) ^{\left( p+1\right) /2}w_{0}^{p+1}d\theta dx,\,\,%
\text{ for }t>0.
\end{eqnarray*}%
Clearly, $t_{0}w_{0}\in \mathbf{M}$ if and only if
\begin{eqnarray*}
\eta \left( t_{0}\right) &=&-\int_{\mathbb{R}^{3}}\mu _{11}\phi _{_{\sqrt{%
s_{\min }}w_{0}}}s_{\min }w_{0}^{2}+\mu _{22}\phi _{_{\sqrt{1-s_{\min }}%
w_{0}}}\left( 1-s_{\min }\right) w_{0}^{2}-2\mu _{12}\phi _{_{\sqrt{%
1-s_{\min }}w_{0}}}\left( 1-s_{\min }\right) w_{0}^{2}dx \\
&=&-\left[ \mu _{11}s_{\min }^{2}+\mu _{22}\left( 1-s_{\min }\right)
^{2}-2\mu _{12}s_{\min }\left( 1-s_{\min }\right) \right] \int_{\mathbb{R}%
^{3}}\phi _{_{w_{0}}}w_{0}^{2}dx \\
&=&-\frac{\mu _{11}\mu _{22}-\mu _{12}^{2}}{\mu _{11}+\mu _{22}+2\mu _{12}}%
\int_{\mathbb{R}^{3}}\phi _{_{w_{0}}}w_{0}^{2}dx,\,\,\text{ for some }t_{0}>0.
\end{eqnarray*}%
Moreover, by Jensen's inequality,%
\begin{eqnarray*}
\frac{1}{2\pi }\int_{0}^{2\pi }\left( 1+2\sqrt{s_{\min }\left( 1-s_{\min
}\right) }\cos \theta \right) ^{\left( p+1\right) /2}d\theta &>&\left( \frac{%
1}{2\pi }\int_{0}^{2\pi }1+2\sqrt{s_{\min }\left( 1-s_{\min }\right) }\cos
\theta d\theta \right) ^{\left( p+1\right) /2} \\
&=&1.
\end{eqnarray*}%
Thus,%
\begin{equation*}
\eta \left( t\right) <\eta _{0}\left( t\right),\, \text{ for }t>0,
\end{equation*}%
where%
\begin{equation*}
\eta _{0}\left( t\right) =t^{-2}\left\Vert w_{0}\right\Vert
_{H^{1}}^{2}-t^{p-3}\int_{\mathbb{R}^{3}}w_{0}^{p+1}dx.
\end{equation*}%
A straightforward evaluation gives
\begin{equation*}
\eta _{0}\left( 1\right) =-\mu _{11}\int_{\mathbb{R}^{3}}\phi
_{_{w_{0}}}w_{0}^{2}dx,\lim_{t\rightarrow 0^{+}}\eta _{0}(t)=\infty \text{
and }\lim_{t\rightarrow \infty }\eta _{0}(t)=0.
\end{equation*}%
Since $1<p<3$ and
\begin{equation*}
\eta _{0}^{\prime }\left( t\right) =t^{-3}\left[ -2\left\Vert
w_{0}\right\Vert _{H^{1}}^{2}+\left( 3-p\right) t^{p-2}\int_{\mathbb{R}%
^{3}}w_{0}^{p+1}dx\right] ,
\end{equation*}%
we find that $\eta _{0}\left( t\right) $ is decreasing when $0<t<\left(
\frac{2\left\Vert w_{0}\right\Vert _{H^{1}}^{2}}{\left( 3-p\right) \int_{%
\mathbb{R}^{3}}w_{0}^{p+1}dx}\right) ^{1/\left( p-1\right) }$ and is
increasing when $t>\left( \frac{2\left\Vert w_{0}\right\Vert _{H^{1}}^{2}}{%
\left( 3-p\right) \int_{\mathbb{R}^{3}}w_{0}^{p+1}dx}\right) ^{1/\left(
p-1\right) }>1.$ This gives
\begin{equation*}
\inf_{t>0}\eta _{0}\left( t\right) =\eta _{0}\left( \left( \frac{2\left\Vert
w_{0}\right\Vert _{H^{1}}^{2}}{\left( 3-p\right) \int_{\mathbb{R}%
^{3}}w_{0}^{p+1}dx}\right) ^{1/\left( p-1\right) }\right) ,
\end{equation*}%
and it follows that
\begin{eqnarray*}
\inf_{t>0}\eta \left( t\right) &\leq &\eta \left( \left( \frac{2\left\Vert
w_{0}\right\Vert _{H^{1}}^{2}}{\left( 3-p\right) \int_{\mathbb{R}%
^{3}}w_{0}^{p+1}dx}\right) ^{1/\left( p-1\right) }\right) <\eta _{0}\left(
\left( \frac{2\left\Vert w_{0}\right\Vert _{H^{1}}^{2}}{\left( 3-p\right)
\int_{\mathbb{R}^{3}}w_{0}^{p+1}dx}\right) ^{1/\left( p-1\right) }\right) \\
&<&\eta _{0}\left( 1\right) \\
&=&-\mu _{11}\int_{\mathbb{R}^{3}}\phi _{_{w_{0}}}w_{0}^{2}dx \\
&<&-\frac{\mu _{11}\mu _{22}-\mu _{12}^{2}}{\mu _{11}+\mu _{22}+2\mu _{12}}%
\int_{\mathbb{R}^{3}}\phi _{_{w_{0}}}w_{0}^{2}dx \\
&=&-\left[ \mu _{11}s_{\min }^{2}+\mu _{22}\left( 1-s_{\min }\right)
^{2}-2\mu _{12}s_{\min }\left( 1-s_{\min }\right) \right] \int_{\mathbb{R}%
^{3}}\phi _{_{w_{0}}}w_{0}^{2}dx \\
&=&-\int_{\mathbb{R}^{3}}\mu _{11}\phi _{_{\sqrt{s_{\min }}w_{0}}}s_{\min
}w_{0}^{2}+\mu _{22}\phi _{_{\sqrt{1-s_{\min }}w_{0}}}\left( 1-s_{\min
}\right) w_{0}^{2}-2\mu _{12}\phi _{_{\sqrt{1-s_{\min }}w_{0}}}\left(
1-s_{\min }\right) w_{0}^{2}dx,
\end{eqnarray*}%
since%
\begin{eqnarray*}
\mu _{11}s_{\min }^{2}+\mu _{22}\left( 1-s_{\min }\right) ^{2}-2\mu
_{12}s_{\min }\left( 1-s_{\min }\right) &=&\frac{\mu _{11}\mu _{22}-\mu
_{12}^{2}}{\mu _{11}+\mu _{22}+2\mu _{12}} \\
&<&\mu _{ii},\,\text{ for }i=1,2.
\end{eqnarray*}%
$\left( i\right) $ $\det \left( \mu _{ij}\right) =\mu _{11}\mu _{22}-\mu
_{12}^{2}>0.$ Since $\lim_{t\rightarrow 0^{+}}\eta (t)=\infty $ and $%
\lim_{t\rightarrow \infty }\eta (t)=0,$ there exist two constants $t_{0}^{+}$
and $t_{0}^{-}>0$ which satisfy
\begin{equation}
t_{0}^{-}<1<\left( \frac{2\left\Vert w_{0}\right\Vert _{H^{1}}^{2}}{\left(
3-p\right) \int_{\mathbb{R}^{3}}w_{0}^{p+1}dx}\right) ^{1/\left( p-1\right)
}<t_{0}^{+}  \label{2-5},
\end{equation}%
such that
\begin{equation*}
\eta \left( t_{0}^{\pm }\right) +\int_{\mathbb{R}^{3}}\mu _{11}\phi _{_{%
\sqrt{s_{\min }}w_{0}}}s_{\min }w_{0}^{2}+\mu _{22}\phi _{_{\sqrt{1-s_{\min }%
}w_{0}}}\left( 1-s_{\min }\right) w_{0}^{2}-2\mu _{12}\phi _{_{\sqrt{%
1-s_{\min }}w_{0}}}\left( 1-s_{\min }\right) w_{0}^{2}dx=0.
\end{equation*}%
That is,
\begin{equation*}
\left( t_{0}^{\pm }\sqrt{s_{\min }}w_{0},t_{0}^{\pm }\sqrt{1-s_{\min }}%
w_{0}\right) \in \mathbf{M}.
\end{equation*}%
By a calculation on the second order derivatives, we find
\begin{eqnarray*}
&&h_{\left( t_{0}^{-}\sqrt{s_{\min }}w_{0},t_{0}^{-}\sqrt{1-s_{\min }}%
w_{0}\right) }^{\prime \prime }\left( 1\right) \\
&=&-2\left\Vert t_{0}^{-}w_{0}\right\Vert _{H^{1}}^{2}+\frac{3-p}{2\pi }%
\int_{\mathbb{R}^{3}}\int_{0}^{2\pi }\left( 1+2\sqrt{s_{\min }\left(
1-s_{\min }\right) }\cos \theta \right) ^{\left( p+1\right) /2}\left\vert
t_{0}^{-}w_{0}\right\vert ^{p+1}d\theta dx \\
&=&\left( t_{0}^{-}\right) ^{5}\eta ^{\prime }\left( t_{0}^{-}\right) \\
&<&0,
\end{eqnarray*}%
and
\begin{eqnarray*}
&&h_{\left( t_{0}^{+}\sqrt{s_{\min }}w_{0},t_{0}^{+}\sqrt{1-s_{\min }}%
w_{0}\right) }^{\prime \prime }\left( 1\right) \\
&=&-2\left\Vert t_{0}^{+}w_{0}\right\Vert _{H^{1}}^{2}+\frac{3-p}{2\pi }%
\int_{\mathbb{R}^{3}}\int_{0}^{2\pi }\left( 1+2\sqrt{s_{\min }\left(
1-s_{\min }\right) }\cos \theta \right) ^{\left( p+1\right) /2}\left\vert
t_{0}^{+}w_{0}\right\vert ^{p+1}d\theta dx \\
&=&\left( t_{0}^{+}\right) ^{5}\eta ^{\prime }\left( t_{0}^{+}\right) \\
&>&0.
\end{eqnarray*}%
This implies that
\begin{equation*}
\left( t_{0}^{\pm }\sqrt{s_{\min }}w_{0},t_{0}^{\pm }\sqrt{1-s_{\min }}%
w_{0}\right) \in \mathbf{M}^{\pm }
\end{equation*}%
and
\begin{eqnarray*}
&&h_{\left( t_{0}^{+}\sqrt{s_{\min }}w_{0},t_{0}^{+}\sqrt{1-s_{\min }}%
w_{0}\right) }^{\prime }\left( t\right) \\
&=&t^{3}\left( \eta (t)+\int_{\mathbb{R}^{3}}\mu _{11}\phi _{_{\sqrt{s_{\min
}}w_{0}}}s_{\min }w_{0}^{2}+\mu _{22}\phi _{_{\sqrt{1-s_{\min }}%
w_{0}}}\left( 1-s_{\min }\right) w_{0}^{2}-2\mu _{12}\phi _{_{\sqrt{%
1-s_{\min }}w_{0}}}\left( 1-s_{\min }\right) w_{0}^{2}dx\right) .
\end{eqnarray*}%
Clearly,
\begin{equation*}
h_{\left( \sqrt{s_{\min }}w_{0},\sqrt{1-s_{\min }}w_{0}\right) }^{\prime
}\left( t\right) >0\text{ for all }t\in \left( 0,t_{0}^{-}\right) \cup
\left( t_{0}^{+},\infty \right)
\end{equation*}%
and
\begin{equation*}
h_{\left( \sqrt{s_{\min }}w_{0},\sqrt{1-s_{\min }}w_{0}\right) }^{\prime
}\left( t\right) <0\text{ for all }t\in \left( t_{0}^{-},t_{0}^{+}\right) .
\end{equation*}%
As a result,
\begin{equation*}
J\left( t_{0}^{-}\sqrt{s_{\min }}w_{0},t_{0}^{-}\sqrt{1-s_{\min }}%
w_{0}\right) =\sup_{0\leq t\leq t_{0}^{+}}J\left( t\sqrt{s_{\min }}w_{0},t%
\sqrt{1-s_{\min }}w_{0}\right)
\end{equation*}%
and%
\begin{equation*}
J\left( t_{0}^{+}\sqrt{s_{\min }}w_{0},t_{0}^{+}\sqrt{1-s_{\min }}%
w_{0}\right) =\inf_{t\geq t_{0}^{-}}J\left( t\sqrt{s_{\min }}w_{0},t\sqrt{%
1-s_{\min }}w_{0}\right) .
\end{equation*}%
Thus,
\begin{equation*}
J\left( t_{0}^{+}\sqrt{s_{\min }}w_{0},t_{0}^{+}\sqrt{1-s_{\min }}%
w_{0}\right) <J\left( t_{0}^{-}\sqrt{s_{\min }}w_{0},t_{0}^{-}\sqrt{%
1-s_{\min }}w_{0}\right) .
\end{equation*}%
Since $t_{0}^{+}>\left( \frac{2\left\Vert w_{0}\right\Vert _{H^{1}}^{2}}{%
\left( 3-p\right) \int_{\mathbb{R}^{3}}w_{0}^{p+1}dx}\right) ^{1/\left(
p-1\right) }>1$ and%
\begin{eqnarray*}
J\left( t\sqrt{s_{\min }}w_{0},t\sqrt{1-s_{\min }}w_{0}\right) &<&\frac{t^{2}%
}{2}\left\Vert w_{0}\right\Vert _{H^{1}}^{2}+\frac{t^{4}\left( \mu _{11}\mu
_{22}-\mu _{12}^{2}\right) }{4\left( \mu _{11}+\mu _{22}+2\mu _{12}\right) }%
\int_{\mathbb{R}^{3}}\phi _{w_{0}}^{2}w_{0}dx \\
&&-\frac{t^{p}}{p+1}\int_{\mathbb{R}^{3}}\left\vert w_{0}\right\vert ^{p+1}dx
\\
&<&\frac{t^{2}}{2}\left\Vert w_{0}\right\Vert _{H^{1}}^{2}+\frac{\mu
_{ii}t^{4}}{4}\int_{\mathbb{R}^{3}}\phi _{w_{0}}^{2}w_{0}dx-\frac{t^{p}}{p+1}%
\int_{\mathbb{R}^{3}}\left\vert w_{0}\right\vert ^{p+1}dx,
\end{eqnarray*}%
we have
\begin{eqnarray*}
J\left( t_{0}^{-}\sqrt{s_{\min }}w_{0},t_{0}^{-}\sqrt{1-s_{\min }}%
w_{0}\right) &=&\sup_{0\leq t\leq t_{0}^{+}}J\left( t\sqrt{s_{\min }}w_{0},t%
\sqrt{1-s_{\min }}w_{0}\right) \\
&<&\frac{1}{2}\left\Vert w_{0}\right\Vert _{H^{1}}^{2}+\frac{\mu _{11}}{4}%
\int_{\mathbb{R}^{3}}\phi _{w_{0}}^{2}w_{0}dx-\frac{1}{p+1}\int_{\mathbb{R}%
^{3}}\left\vert w_{0}\right\vert ^{p+1}dx \\
&=&\beta _{\mu _{11}},
\end{eqnarray*}%
implying that $\left( t_{0}^{-}\sqrt{s_{\min }}w_{0},t_{0}^{-}\sqrt{%
1-s_{\min }}w_{0}\right) \in \mathbf{M}^{\left( 1\right) }.$ Moreover,
\begin{eqnarray*}
J\left( t\sqrt{s_{\min }}w_{0},t\sqrt{1-s_{\min }}w_{0}\right) &<&\frac{t^{2}%
}{2}\left\Vert w_{0}\right\Vert _{H^{1}}^{2}+\frac{\mu _{11}t^{4}}{4}\int_{%
\mathbb{R}^{3}}\phi _{w_{0}}^{2}w_{0}dx-\frac{t^{p+1}}{p+1}\int_{\mathbb{R}%
^{3}}\left\vert w_{0}\right\vert ^{p+1}dx \\
&=&t^{4}\left[ \xi \left( t\right) +\frac{\mu _{11}}{4}\int_{\mathbb{R}%
^{3}}\phi _{w_{0}}^{2}w_{0}dx\right] ,
\end{eqnarray*}%
where
\begin{equation*}
\xi \left( t\right) =\frac{t^{-2}}{2}\left\Vert w_{0}\right\Vert
_{H^{1}}^{2}-\frac{t^{p-3}}{p+1}\int_{\mathbb{R}^{3}}\left\vert
w_{0}\right\vert ^{p+1}dx.
\end{equation*}%
Clearly, $J\left( t_{0}\sqrt{s_{\min }}w_{0},t_{0}\sqrt{1-s_{\min }}%
w_{0}\right) <0$ if
\begin{equation*}
\xi \left( t_{0}\right) +\frac{\mu _{11}}{4}\int_{\mathbb{R}^{3}}\phi
_{w_{0}}^{2}w_{0}dx\leq 0\text{ for some }t_{0}>0.
\end{equation*}%
It is not difficult to observe that
\begin{equation*}
\xi \left( \hat{t}_{0}\right) =0,\ \ \lim_{t\rightarrow 0^{+}}\xi (t)=\infty
\ \text{ and }\ \lim_{t\rightarrow \infty }\xi (t)=0,
\end{equation*}%
where $\hat{t}_{0}=\left( \frac{\left( p+1\right) \left\Vert
w_{0}\right\Vert _{H^{1}}^{2}}{2\int_{\mathbb{R}^{3}}\left\vert
w_{0}\right\vert ^{p+1}dx}\right) ^{1/\left( p-1\right) }.$ Considering the
derivative of $\xi (t)$, we find
\begin{eqnarray*}
\xi ^{\prime }\left( t\right) &=&-t^{-3}\left\Vert w_{0}\right\Vert
_{H^{1}}^{2}+\frac{3-p}{p+1}t^{p-4}\int_{\mathbb{R}^{3}}\left\vert
w_{0}\right\vert ^{p+1}dx \\
&=&t^{-3}\left[ \frac{3-p}{p+1}t^{p-1}\int_{\mathbb{R}^{3}}\left\vert
w_{0}\right\vert ^{p+1}dx-\left\Vert w_{0}\right\Vert _{H^{1}}^{2}\right] ,
\end{eqnarray*}%
which implies that $\xi \left( t\right) $ is decreasing when $0<t<\left(
\frac{\left( p+1\right) \left\Vert w_{0}\right\Vert _{H^{1}}^{2}}{\left(
3-p\right) \int_{\mathbb{R}^{3}}\left\vert w_{0}\right\vert ^{p+1}dx}\right)
^{1/\left( p-1\right) }$ and is increasing when $t>\left( \frac{\left(
p+1\right) \left\Vert w_{0}\right\Vert _{H^{1}}^{2}}{\left( 3-p\right) \int_{%
\mathbb{R}^{3}}\left\vert w_{0}\right\vert ^{p+1}dx}\right) ^{1/\left(
p-1\right) }.$ Then%
\begin{equation*}
t_{0}^{-}<1<\left( \frac{2\left\Vert w_{0}\right\Vert _{H^{1}}^{2}}{\left(
3-p\right) \int_{\mathbb{R}^{3}}\left\vert w_{0}\right\vert ^{p+1}dx}\right)
^{1/\left( p-1\right) }<\left( \frac{\left( p+1\right) \left\Vert
w_{0}\right\Vert _{H^{1}}^{2}}{\left( 3-p\right) \int_{\mathbb{R}%
^{3}}\left\vert w_{0}\right\vert ^{p+1}dx}\right) ^{1/\left( p-1\right) }
\end{equation*}%
and%
\begin{eqnarray*}
\inf_{t>0}\xi \left( t\right) &=&\xi \left[ \left( \frac{\left( p+1\right)
\left\Vert w_{0}\right\Vert _{H^{1}}^{2}}{\left( 3-p\right) \int_{\mathbb{R}%
^{3}}\left\vert w_{0}\right\vert ^{p+1}dx}\right) ^{1/\left( p-1\right) }%
\right] \\
&=&-\frac{p-1}{2\left( 3-p\right) }\left( \frac{\left( 3-p\right) \int_{%
\mathbb{R}^{3}}\left\vert w_{0}\right\vert ^{p+1}dx}{\left( p+1\right)
\left\Vert w_{0}\right\Vert _{H^{1}}^{2}}\right) ^{2/\left( p-1\right)
}\left\Vert w_{0}\right\Vert _{H^{1}}^{2} \\
&<&-\frac{p-1}{2\left( 3-p\right) }\left\Vert w_{0}\right\Vert _{H^{1}}^{2}
\\
&<&-\frac{4\mu _{11}}{3\sqrt{3}\pi \lambda ^{\frac{3}{2}}}\left\Vert
w_{0}\right\Vert _{H^{1}}^{4} \\
&<&-\frac{\mu _{11}}{4}\int_{\mathbb{R}^{3}}\phi _{w_{0}}^{2}w_{0}dx,
\end{eqnarray*}%
which, subsequently, yields
\begin{equation*}
J\left( t_{0}^{+}\sqrt{s_{\min }}w_{0},t_{0}^{+}\sqrt{1-s_{\min }}%
w_{0}\right) =\inf_{t\geq t_{0}^{-}}J\left( t\sqrt{s_{\min }}w_{0},t\sqrt{%
1-s_{\min }}w_{0}\right) <0,
\end{equation*}%
indicating that $\left( t_{0}^{+}\sqrt{s_{\min }}w_{0},t_{0}^{+}\sqrt{%
1-s_{\min }}w_{0}\right) \in \mathbf{M}^{\left( 2\right) }.$\newline
$\left( ii\right) $ $\det \left( \mu _{ij}\right) =\mu _{11}\mu _{22}-\mu
_{12}^{2}\leq 0.$ The proof is similar to the argument used in part $\left(
i\right) $ and is therefore omitted here. This completes the proof.
\end{proof}

\vspace{3mm}
Define
\begin{equation*}
\alpha ^{-}:=\inf_{\left( u,v\right) \in \mathbf{M}^{\left( 1\right)
}}J\left( u,v\right) \text{ for }1<p<3.
\end{equation*}%
Clearly, $0<\frac{p-1}{4\left( p+1\right) }C_{\mu _{12}}^{2}<\alpha
^{-}<\beta _{\mu _{11}}$ for $0<\mu _{11}<\Lambda _{0}.$

\bigskip

\textbf{We are now ready to prove Theorem \ref{t2}.} By Lemmas \ref{g5}, \ref%
{g7}, \ref{g4} and the Ekeland variational principle, there exists a
minimizing sequence $\left\{ \left( u_{n},v_{n}\right) \right\} \subset
\mathbf{M}^{\left( 1\right) }$ such that
\begin{equation}
J\left( u_{n},v_{n}\right) =\alpha ^{-}+o\left( 1\right) \text{ and }%
J^{\prime }\left( u_{n},v_{n}\right) =o\left( 1\right) \text{ in }H^{-1}.
\label{18-0}
\end{equation}%
Since $\left\{ \left( u_{n},v_{n}\right) \right\} $ is bounded, there exists
a convergent subsequence of $\left\{ \left( u_{n},v_{n}\right) \right\} $
(which will also be denoted by $\left\{ \left( u_{n},v_{n}\right) \right\} $ for
convenience) such that as $n\rightarrow \infty $,
\begin{equation}
\left( u_{n},v_{n}\right) \rightharpoonup \left( u_{0},v_{0}\right) \text{
weakly in }H,  \label{12-1}
\end{equation}%
where $\left( u_{0},v_{0}\right) \in H$. By (\ref{12-1}) and Sobolev compact
embedding, we obtain
\begin{eqnarray}
\left( u_{n},v_{n}\right) &\rightarrow &\left( u_{0},v_{0}\right) \text{
strongly in }L_{loc}^{p}\left( \mathbb{R}^{3}\right) \times
L_{loc}^{p}\left( \mathbb{R}^{3}\right) ,  \label{12-2} \\
\left( u_{n},v_{n}\right) &\rightarrow &\left( u_{0},v_{0}\right) \text{
a.e. in }\mathbb{R}^{3}.  \label{12-3}
\end{eqnarray}%
Now we claim that there exist a subsequence $\left\{ \left(
u_{n},v_{n}\right) \right\} _{n=1}^{\infty }$ and a sequence $%
\{x_{n}\}_{n=1}^{\infty }\subset \mathbb{R}^{3}$ such that
\begin{equation}
\int_{B^{N}\left( x_{n},R\right) }\left\vert \left( u_{n},v_{n}\right)
\right\vert ^{2}dx\geq d_{0}>0\text{ for all }n\in \mathbb{N},  \label{12-5}
\end{equation}%
where $d_{0}$ and $R$ are positive constants that are independent of $n.$
Suppose the contrary is true. Then, for all $R>0$,
\begin{equation*}
\sup_{x\in \mathbb{R}^{N}}\int_{B^{N}\left( x_{n},R\right) }\left\vert
\left( u_{n},v_{n}\right) \right\vert ^{2}dx\rightarrow 0\text{ as }%
n\rightarrow \infty .
\end{equation*}%
Thus, applying the argument of Lemma I.1 in \cite{Li1} (see also \cite%
{Wi}) gives
\begin{equation}
\int_{\mathbb{R}^{N}}\left\vert u_{n}\right\vert ^{r}+\left\vert
v_{n}\right\vert ^{r}dx\rightarrow 0\text{ as }n\rightarrow \infty ,
\label{12-4}
\end{equation}%
for all $2<r<2^{\ast }.$ Then we have
\begin{equation*}
\int_{\mathbb{R}^{3}}\int_{0}^{2\pi }\left\vert u_{n}+e^{i\theta
}v_{n}\right\vert ^{p+1}d\theta dx\rightarrow 0\text{ as }n\rightarrow \infty
\end{equation*}%
and
\begin{equation*}
\int_{\mathbb{R}^{3}}\mu _{11}\phi _{_{u_{n}}}u_{n}^{2}+\mu _{22}\phi
_{_{v_{n}}}v_{n}^{2}-2\mu _{12}\phi _{_{v_{n}}}u_{n}^{2}dx\rightarrow 0\text{
as }n\rightarrow \infty ,
\end{equation*}%
implying %
\begin{eqnarray*}
\alpha ^{-}+o\left( 1\right) &=&J\left( u_{n},v_{n}\right) \\
&=&-\frac{1}{4}\int_{\mathbb{R}^{3}}\mu _{11}\phi _{_{u_{n}}}u_{n}^{2}+\mu
_{22}\phi _{_{v_{n}}}v_{n}^{2}-2\mu _{12}\phi _{_{v_{n}}}u_{n}^{2}dx \\
&&+\frac{p-1}{4\left( p+1\right) \pi }\int_{\mathbb{R}^{3}}\int_{0}^{2\pi
}\left\vert u_{n}+e^{i\theta }v_{n}\right\vert ^{p+1}d\theta dx \\
&=&o\left( 1\right);
\end{eqnarray*}%
this contradicts $\alpha ^{-}>0.$ Let $\left( \overline{u}_{n}\left(
x\right) ,\overline{v}_{n}\left( x\right) \right) =\left( u_{n}\left(
x-x_{n}\right) ,v_{n}\left( x-x_{n}\right) \right) .$ Clearly, $\left\{
\left( \overline{u}_{n},\overline{v}_{n}\right) \right\} \subset \mathbf{M}%
^{\left( 1\right) }$ such that
\begin{equation}
J\left( \overline{u}_{n},\overline{v}_{n}\right) =\alpha ^{-}+o\left(
1\right) \text{ and }J^{\prime }\left( \overline{u}_{n},\overline{v}%
_{n}\right) =o\left( 1\right) \text{ in }H^{-1}.  \label{12-7}
\end{equation}%
Since $\left\{ \left( \overline{u}_{n},\overline{v}_{n}\right) \right\} $
is also bounded, there exists a convergent subsequence of $\left\{ \left(
\overline{u}_{n},\overline{v}_{n}\right) \right\} $ and $\left(
u_{0}^{\left( 1\right) },v_{0}^{\left( 1\right) }\right) \in H$ such that as
$n\rightarrow \infty $,
\begin{equation}
\left( \overline{u}_{n},\overline{v}_{n}\right) \rightharpoonup \left(
u_{0}^{\left( 1\right) },v_{0}^{\left( 1\right) }\right) \text{ weakly in }H.
\label{15-1}
\end{equation}%
By (\ref{12-1}) and Sobolev compact embedding, we obtain
\begin{eqnarray}
\left( \overline{u}_{n},\overline{v}_{n}\right) &\rightarrow &\left(
u_{0}^{\left( 1\right) },v_{0}^{\left( 1\right) }\right) \text{ strongly in }%
L_{loc}^{p}\left( \mathbb{R}^{3}\right) \times L_{loc}^{p}\left( \mathbb{R}%
^{3}\right) ,  \label{15-2} \\
\left( \overline{u}_{n},\overline{v}_{n}\right) &\rightarrow &\left(
u_{0}^{\left( 1\right) },v_{0}^{\left( 1\right) }\right) \text{ a.e. in }%
\mathbb{R}^{3}.  \label{15-3}
\end{eqnarray}%
Moreover, by $\left( \ref{12-5}\right) $ and $(\ref{12-7})-(\ref{15-3})$,
\begin{equation*}
\int_{B^{N}\left( R\right) }\left\vert \left( u_{0}^{\left( 1\right)
},v_{0}^{\left( 1\right) }\right) \right\vert ^{2}dx\geq d_{0}>0,
\end{equation*}%
and the function $\left( u_{0}^{\left( 1\right) },v_{0}^{\left( 1\right)
}\right) $ is a nontrivial solution of System $\left( E\right) .$ On the
other hand, we have $\left( u_{0}^{\left( 1\right) },v_{0}^{\left( 1\right)
}\right) \in \mathbf{M}.$ By Fatou's Lemma,%
\begin{equation*}
\left\Vert \left( u_{0}^{\left( 1\right) },v_{0}^{\left( 1\right) }\right)
\right\Vert _{H}^{2}\leq \liminf_{n\rightarrow \infty }\left\Vert \left(
\overline{u}_{n},\overline{v}_{n}\right) \right\Vert _{H}^{2}.
\end{equation*}%
Suppose that
\begin{equation}
\left\Vert \left( u_{0}^{\left( 1\right) },v_{0}^{\left( 1\right) }\right)
\right\Vert _{H}^{2}<\liminf_{n\rightarrow \infty }\left\Vert \left(
\overline{u}_{n},\overline{v}_{n}\right) \right\Vert _{H}^{2}.  \label{15-4}
\end{equation}%
Then $\left\Vert \left( u_{0}^{\left( 1\right) },v_{0}^{\left( 1\right)
}\right) \right\Vert _{H}<\left( \frac{3\sqrt{3}\left( p-1\right) \pi
\lambda ^{\frac{3}{2}}}{16\mu _{22}(3-p)}\right) ^{1/2}$. By $\left( \ref%
{2-6-1}\right) $ and Lemma \ref{L2-3}, it follows from the Sobolev inequality
that
\begin{eqnarray*}
h_{\left( u_{0}^{\left( 1\right) },v_{0}^{\left( 1\right) }\right) }^{\prime
\prime }\left( 1\right) &=&-\left( p-2\right) \left\Vert \left(
u_{0}^{\left( 1\right) },v_{0}^{\left( 1\right) }\right) \right\Vert
_{H}^{2}+\lambda \left( 4-p\right) \int_{\mathbb{R}^{3}}\phi _{u_{0}^{\left(
1\right) },v_{0}^{\left( 1\right) }}\left( \left[ u_{0}^{\left( 1\right) }%
\right] ^{2}+\left[ v_{0}^{\left( 1\right) }\right] ^{2}\right) dx \\
&\leq &\left\Vert \left( u_{0}^{\left( 1\right) },v_{0}^{\left( 1\right)
}\right) \right\Vert _{H}^{2}\left[ \frac{\lambda (4-p)}{\overline{S}%
^{2}S_{12/5}^{4}}\left\Vert \left( u_{0}^{\left( 1\right) },v_{0}^{\left(
1\right) }\right) \right\Vert _{H}^{2}-\left( p-2\right) \right] \\
&<&\left\Vert \left( u_{0}^{\left( 1\right) },v_{0}^{\left( 1\right)
}\right) \right\Vert _{H}^{2}\left( \frac{16\mu _{11}(3-p)}{3\sqrt{3}\pi
\lambda ^{\frac{3}{2}}}\frac{3\sqrt{3}\left( p-1\right) \pi \lambda ^{\frac{3%
}{2}}}{16\mu _{11}(3-p)}-\left( p-1\right) \right) =0,
\end{eqnarray*}%
this indicates that $\left( u_{0}^{\left( 1\right) },v_{0}^{\left( 1\right)
}\right) \in \mathbf{M}^{-}$ and $J\left( u_{0}^{\left( 1\right)
},v_{0}^{\left( 1\right) }\right) \geq \alpha ^{-}.$ Let $\left(
w_{n},z_{n}\right) =\left( \overline{u}_{n}-u_{0}^{\left( 1\right) },%
\overline{v}_{n}-v_{0}^{\left( 1\right) }\right) .$ Then, by $\left( \ref%
{15-1}\right) $ and $\left( \ref{15-4}\right) ,$ there exists $c_{0}>0$ such
that%
\begin{equation*}
c_{0}\leq \left\Vert \left( w_{n},z_{n}\right) \right\Vert
_{H}^{2}=\left\Vert \left( \overline{u}_{n},\overline{v}_{n}\right)
\right\Vert _{H}^{2}-\left\Vert \left( u_{0}^{\left( 1\right)
},v_{0}^{\left( 1\right) }\right) \right\Vert _{H}^{2}+o\left( 1\right) ,
\end{equation*}%
which implies that%
\begin{equation}
\left\Vert \left( w_{n},z_{n}\right) \right\Vert _{H}^{2}<\left( \frac{3%
\sqrt{3}\left( p-1\right) \pi \lambda ^{\frac{3}{2}}}{16\mu _{22}(3-p)}%
\right) ^{1/2},\,\text{ for }n\text{ sufficiently large.}  \label{15-5}
\end{equation}%
On the other hand, the Brezis-Lieb Lemma(cf. \cite{BLi}) gives%
\begin{equation*}
\int_{\mathbb{R}^{3}}\int_{0}^{2\pi }\left\vert \overline{u}_{n}+e^{i\theta }%
\overline{v}_{n}\right\vert ^{p+1}d\theta dx=\int_{\mathbb{R}%
^{3}}\int_{0}^{2\pi }\left\vert w_{n}+e^{i\theta }z_{n}\right\vert
^{p+1}d\theta dx+\int_{\mathbb{R}^{3}}\int_{0}^{2\pi }\left\vert
u_{0}^{\left( 1\right) }+e^{i\theta }v_{0}^{\left( 1\right) }\right\vert
^{p+1}d\theta dx
\end{equation*}%
and%
\begin{eqnarray*}
&&\int_{\mathbb{R}^{3}}\mu _{11}\phi _{_{\overline{u}_{n}}}\overline{u}%
_{n}^{2}+\mu _{22}\phi _{_{\overline{v}_{n}}}\overline{v}_{n}^{2}-2\mu
_{12}\phi _{_{\overline{v}_{n}}}\overline{u}_{n}^{2}dx \\
&=&\int_{\mathbb{R}^{3}}\mu _{11}\phi _{_{w_{n}}}w_{n}^{2}+\mu _{22}\phi
_{_{z_{n}}}z_{n}^{2}-2\mu _{12}\phi _{_{z_{n}}}w_{n}^{2}dx \\
&&+\int_{\mathbb{R}^{3}}\mu _{11}\phi _{_{u_{0}^{\left( 1\right) }}}\left[
u_{0}^{\left( 1\right) }\right] ^{2}+\mu _{22}\phi _{_{v_{0}^{\left(
1\right) }}}\left[ v_{0}^{\left( 1\right) }\right] ^{2}-2\mu _{12}\phi
_{_{v_{0}^{\left( 1\right) }}}\left[ u_{0}^{\left( 1\right) }\right] ^{2}dx.
\end{eqnarray*}%
This implies that%
\begin{equation}
\left\Vert \left( w_{n},z_{n}\right) \right\Vert _{H}^{2}+\int_{\mathbb{R}%
^{3}}\phi _{w_{n},z_{n}}\left( w_{n}^{2}+z_{n}^{2}\right) dx-\int_{\mathbb{R}%
^{3}}\int_{0}^{2\pi }\left\vert w_{n}+e^{i\theta }z_{n}\right\vert
^{p+1}d\theta dx=o\left( 1\right)  \label{15-6}
\end{equation}%
and%
\begin{equation}
J\left( \overline{u}_{n},\overline{v}_{n}\right) =J\left( w_{n},z_{n}\right)
+J\left( u_{0}^{\left( 1\right) },v_{0}^{\left( 1\right) }\right) +o\left(
1\right) .  \label{15-7}
\end{equation}%
Moreover, by $\left( \ref{15-5}\right) $ and $\left( \ref{15-6}\right) ,$
there exists $s_{n}=1+o\left( 1\right) $ such that%
\begin{equation*}
\left\Vert \left( s_{n}w_{n},s_{n}z_{n}\right) \right\Vert _{H}^{2}+\int_{%
\mathbb{R}^{3}}\phi _{s_{n}w_{n},s_{n}z_{n}}\left(
s_{n}^{2}w_{n}^{2}+s_{n}^{2}z_{n}^{2}\right) dx-\int_{\mathbb{R}%
^{3}}\int_{0}^{2\pi }\left\vert s_{n}w_{n}+e^{i\theta }s_{n}z_{n}\right\vert
^{p+1}d\theta dx=0
\end{equation*}%
and%
\begin{equation*}
\left\Vert \left( s_{n}w_{n},s_{n}z_{n}\right) \right\Vert _{H}^{2}<\left(
\frac{3\sqrt{3}\left( p-1\right) \pi \lambda ^{\frac{3}{2}}}{16\mu _{22}(3-p)%
}\right) ^{1/2}\text{ for }n\text{ sufficiently large.}
\end{equation*}%
Hence,%
\begin{equation*}
h_{\lambda ,\left( s_{n}w_{n},s_{n}z_{n}\right) }^{\prime \prime }\left(
1\right) =-\left( p-1\right) \left\Vert \left( s_{n}w_{n},s_{n}z_{n}\right)
\right\Vert _{H}^{2}+\lambda \left( 3-p\right) \left(
s_{n}w_{n},s_{n}z_{n}\right) <0,
\end{equation*}%
implying that $J\left( s_{n}w_{n},s_{n}z_{n}\right) \geq \frac{1}{2}%
\alpha ^{-}$  when  $n$  is sufficiently large. Therefore,%
\begin{equation*}
\alpha ^{-}+o\left( 1\right) =J\left( \overline{u}_{n},\overline{v}%
_{n}\right) \geq \frac{3}{2}\alpha ^{-},\,\text{ for }n\text{ sufficiently
large,}
\end{equation*}%
which is a contradiction. Thus, we can conclude that $\left\Vert \left( u_{0}^{\left(
1\right) },v_{0}^{\left( 1\right) }\right) \right\Vert
_{H}^{2}=\liminf_{n\rightarrow \infty }\left\Vert \left( \overline{u}_{n},%
\overline{v}_{n}\right) \right\Vert _{H}^{2},$ then
\begin{equation*}
\left( \overline{u}_{n},\overline{v}_{n}\right) \rightarrow \left(
u_{0}^{\left( 1\right) },v_{0}^{\left( 1\right) }\right) \text{ strongly in }%
H,
\end{equation*}%
and $J\left( u_{0}^{\left( 1\right) },v_{0}^{\left( 1\right) }\right)
=\alpha ^{-},$ then so is $\left( \left\vert u_{0}^{\left( 1\right)
}\right\vert ,\left\vert v_{0}^{\left( 1\right) }\right\vert \right) .$
Then, by Lemma \ref{g7}, we may assume that $\left( u_{0}^{\left( 1\right)
},v_{0}^{\left( 1\right) }\right) $ is a positive nontrivial critical point
of $J$. Moreover, by Lemma \ref{l5} and $\alpha ^{-}<\beta _{\mu _{11}},$ we
have $u_{0}^{\left( 1\right) }\neq 0$ and $v_{0}^{\left( 1\right) }\neq 0.$
Moreover,
\begin{equation*}
\left\Vert \left( u_{0}^{\left( 1\right) },v_{0}^{\left( 1\right) }\right)
\right\Vert _{H}\leq \left( \frac{3\sqrt{3}\left( p-1\right) \pi \lambda ^{%
\frac{3}{2}}}{16\mu _{22}(3-p)}\right) ^{1/2}\rightarrow 0\text{ as }\mu
_{22}\rightarrow \infty .
\end{equation*}%
We complete the proof.

\section{Positive ground state solutions}

Define
\begin{equation*}
\mathbb{A}:=\left\{ \left( u,v\right) \in H\setminus \left\{ \left(
0,0\right) \right\} :\left( u,v\right) \text{ is a solution of System }(E)%
\text{ with }J\left( u,v\right) <D_{0}\right\} ,
\end{equation*}%
where $D_{0}=\frac{A\left( p\right) \left( p-1\right) }{2\left( p+1\right) }%
\left( \frac{2S_{p+1}^{p+1}}{3-p}\right) ^{2/(p-1)}.$ Clearly, $\mathbb{A}%
\subset \mathbf{M}\left[ D_{0}\right] .$ Let
\begin{equation*}
\overline{\Lambda }_{0}:=\left\{
\begin{array}{ll}
\frac{3\sqrt{3}(p+1)^{2}(p-1)^{1/2}\pi \lambda ^{3/2}}{%
8(5-p)^{2}(3-p)^{1/2} }\left( \frac{3-p}{2S_{p+1}^{p+1}}%
\right) ^{2/(p-1)}, & \text{ if }2\leq p<\frac{\sqrt{73}-2}{3}, \\
\infty , & \text{ if }\frac{\sqrt{73}-2}{3}\leq p<3.%
\end{array}%
\right.
\end{equation*}%
Then we have the following results.

\begin{proposition}
\label{t6}Suppose that $2\leq p<3$ and $\mu _{ij}>0.$ Then we have\newline
$\left( i\right) $ if $2<p<3$ and $\mu _{11}\mu _{22}-\mu _{12}^{2}\geq 0,$
then for each $\mu _{ii}\in \left( 0,\overline{\Lambda }_{0}\right) ,$ we
have $\mathbb{A}\subset \mathbf{M}^{-};$\newline
$\left( ii\right) $ if $p=2,$ then for every $\mu _{ii}\in \left( 0,%
\overline{\Lambda }_{0}\right) ,$ we have $\mathbb{A}\subset \mathbf{M}^{-}.$
\end{proposition}

\begin{proof}
Let $\left( u_{0},v_{0}\right) \in \mathbb{A}$ be a nontrivial solution of
System $\left( E\right) .$ Then $\left( u_{0},v_{0}\right) $ satisfies the
Nehari identity:%
\begin{equation}
\left\Vert \left( u_{0},v_{0}\right) \right\Vert _{H}^{2}+\int_{\mathbb{R}%
^{3}}\mu _{11}\phi _{_{u_{0}}}u_{0}^{2}+\mu _{22}\phi
_{_{v_{0}}}v_{0}^{2}-2\mu _{12}\phi _{_{v_{0}}}u_{0}^{2}dx-\frac{1}{2\pi }%
\int_{\mathbb{R}^{3}}\int_{0}^{2\pi }\left\vert tu_{0}+e^{i\theta
}tv_{0}\right\vert ^{p+1}d\theta dx=0.  \label{2-6}
\end{equation}%
Following the argument of \cite[Lemma 3.1]{DM1}, it is not difficult to
verify that solution $\left( u_{0},v_{0}\right) $ also satisfies the
following Pohozaev type identity:%
\begin{eqnarray}
&&\frac{1}{2}\left( \int_{\mathbb{R}^{3}}\left\vert \nabla u_{0}\right\vert
^{2}dx+\int_{\mathbb{R}^{3}}\left\vert \nabla v_{0}\right\vert ^{2}dx\right)
+\frac{3}{2}\left( \int_{\mathbb{R}^{3}}\lambda u_{0}^{2}dx+\int_{\mathbb{R}%
^{3}}\lambda v_{0}^{2}dx\right)  \notag \\
&&+\frac{5}{4}\int_{\mathbb{R}^{3}}\mu _{11}\phi _{_{u_{0}}}u_{0}^{2}+\mu
_{22}\phi _{_{v_{0}}}v_{0}^{2}-2\mu _{12}\phi _{_{v_{0}}}u_{0}^{2}dx  \notag
\\
&=&\frac{3}{2\pi \left( p+1\right) }\int_{\mathbb{R}^{3}}\int_{0}^{2\pi
}\left\vert tu_{0}+e^{i\theta }tv_{0}\right\vert ^{p+1}d\theta dx.
\label{2-7}
\end{eqnarray}%
Assume that%
\begin{eqnarray}
J\left( u_{0},v_{0}\right) &=&\frac{1}{2}\left\Vert \left(
u_{0},v_{0}\right) \right\Vert _{H}^{2}+\frac{1}{4}\int_{\mathbb{R}^{3}}\mu
_{11}\phi _{_{u_{0}}}u_{0}^{2}+\mu _{22}\phi _{_{v_{0}}}v_{0}^{2}-2\mu
_{12}\phi _{_{v_{0}}}u_{0}^{2}dx  \notag \\
&&-\frac{1}{2\pi \left( p+1\right) }\int_{\mathbb{R}^{3}}\int_{0}^{2\pi
}\left\vert tu_{0}+e^{i\theta }tv_{0}\right\vert ^{p+1}d\theta dx  \notag \\
&=&\theta .  \label{2-13}
\end{eqnarray}%
Using $\left( \ref{2-6}\right) -\left( \ref{2-13}\right) $, we have%
\begin{eqnarray*}
\theta &=&\frac{p-1}{2\left( p+1\right) }\left( \int_{\mathbb{R}%
^{3}}\left\vert \nabla u_{0}\right\vert ^{2}dx+\int_{\mathbb{R}%
^{3}}\left\vert \nabla v_{0}\right\vert ^{2}dx\right) +\frac{p-1}{2\left(
p+1\right) }\left( \int_{\mathbb{R}^{3}}\lambda u_{0}^{2}dx+\int_{\mathbb{R}%
^{3}}\lambda v_{0}^{2}dx\right) \\
&&-\frac{3-p}{4\left( p+1\right) }\int_{\mathbb{R}^{3}}\mu _{11}\phi
_{_{u_{0}}}u_{0}^{2}+\mu _{22}\phi _{_{v_{0}}}v_{0}^{2}-2\mu _{12}\phi
_{_{v_{0}}}u_{0}^{2}dx
\end{eqnarray*}%
and%
\begin{eqnarray*}
\int_{\mathbb{R}^{3}}\left\vert \nabla u_{0}\right\vert ^{2}dx+\int_{\mathbb{%
R}^{3}}\left\vert \nabla v_{0}\right\vert ^{2}dx &=&\frac{3(p-1)}{5-p}\left(
\int_{\mathbb{R}^{3}}\lambda u_{0}^{2}dx+\int_{\mathbb{R}^{3}}\lambda
v_{0}^{2}dx\right) \\
&&+\frac{5p-7}{2\left( 5-p\right) }\int_{\mathbb{R}^{3}}\mu _{11}\phi
_{_{u_{0}}}u_{0}^{2}+\mu _{22}\phi _{_{v_{0}}}v_{0}^{2}-2\mu _{12}\phi
_{_{v_{0}}}u_{0}^{2}dx,
\end{eqnarray*}%
which in turn gives%
\begin{equation}
\theta =\frac{p-1}{5-p}\left( \int_{\mathbb{R}^{3}}\lambda u_{0}^{2}dx+\int_{%
\mathbb{R}^{3}}\lambda v_{0}^{2}dx\right) +\frac{p-2}{5-p}\int_{\mathbb{R}%
^{3}}\mu _{11}\phi _{_{u_{0}}}u_{0}^{2}+\mu _{22}\phi
_{_{v_{0}}}v_{0}^{2}-2\mu _{12}\phi _{_{v_{0}}}u_{0}^{2}dx.  \label{2-9}
\end{equation}%
$\left( i\right) $ Since $2<p<3$  and $\mu
_{11}\mu _{22}-\mu _{12}^{2}\geq 0,$ equation $\left( \ref{2-9}\right) $ gives%
\begin{eqnarray}
\theta &=&\frac{p-1}{5-p}\left( \int_{\mathbb{R}^{3}}\lambda
u_{0}^{2}dx+\int_{\mathbb{R}^{3}}\lambda v_{0}^{2}dx\right) +\frac{p-2}{5-p}%
\int_{\mathbb{R}^{3}}\mu _{11}\phi _{_{u_{0}}}u_{0}^{2}+\mu _{22}\phi
_{_{v_{0}}}v_{0}^{2}-2\mu _{12}\phi _{_{v_{0}}}u_{0}^{2}dx  \notag \\
&\geq &\frac{p-1}{5-p}\left( \int_{\mathbb{R}^{3}}\lambda u_{0}^{2}dx+\int_{%
\mathbb{R}^{3}}\lambda v_{0}^{2}dx\right) >0;  \label{2-9-1}
\end{eqnarray}%
here we have used the inequality,%
\begin{equation*}
\int_{\mathbb{R}^{3}}\phi _{_{v_{0}}}u_{0}^{2}dx\leq \frac{\mu _{11}}{2\mu
_{12}}\int_{\mathbb{R}^{3}}\phi _{_{u_{0}}}u_{0}^{2}+\frac{\mu _{12}}{2\mu
_{11}}\int_{\mathbb{R}^{3}}\phi _{_{v_{0}}}v_{0}^{2}.
\end{equation*}%
Moreover, using Hardy-Littlewood-Sobolev and Gagliardo-Nirenberg
inequalities, we have,%
\begin{equation}
\int_{\mathbb{R}^{3}}\phi _{u_{0}}u_{0}^{2}dx\leq \frac{16}{3\sqrt{3}\pi
\lambda ^{\frac{3}{2}}}\left( \int_{\mathbb{R}^{3}}\lambda
u_{0}^{2}dx\right) ^{\frac{3}{2}}\left( \int_{\mathbb{R}^{3}}|\nabla
u_{0}|^{2}dx\right) ^{\frac{1}{2}}  \label{2-8-1}
\end{equation}%
and%
\begin{equation}
\int_{\mathbb{R}^{3}}\phi _{v_{0}}v_{0}^{2}dx\leq \frac{16}{3\sqrt{3}\pi
\lambda ^{\frac{3}{2}}}\left( \int_{\mathbb{R}^{3}}\lambda
v_{0}^{2}dx\right) ^{\frac{3}{2}}\left( \int_{\mathbb{R}^{3}}|\nabla
v_{0}|^{2}dx\right) ^{\frac{1}{2}}.  \label{2-8-2}
\end{equation}%
We now rewrite $\left( \ref{2-6}\right) -\left( \ref{2-13}\right) $ using the following notations,%
\begin{equation*}
\begin{array}{ll}
z_{1}=\int_{\mathbb{R}^{3}}|\nabla u_{0}|^{2}dx+\int_{\mathbb{R}^{3}}|\nabla
v_{0}|^{2}dx, & z_{2}=\int_{\mathbb{R}^{3}}\lambda u_{0}^{2}dx+\int_{\mathbb{%
R}^{3}}\lambda v_{0}^{2}dx, \\
z_{3}=\int_{\mathbb{R}^{3}}\phi _{u_{0}}u_{0}^{2}dx, & z_{4}=\int_{\mathbb{R}%
^{3}}\phi _{v_{0}}v_{0}^{2}dx, \\
z_{5}=\int_{\mathbb{R}^{3}}\phi _{v_{0}}u_{0}^{2}dx, & z_{6}=\frac{1}{2\pi }%
\int_{\mathbb{R}^{3}}\int_{0}^{2\pi }\left\vert tu_{0}+e^{i\theta
}tv_{0}\right\vert ^{p+1}d\theta dx,%
\end{array}%
\end{equation*}%
and arrive at the linear system,%
\begin{equation}
\left\{
\begin{array}{l}
\frac{1}{2}z_{1}+\frac{1}{2}z_{2}+\frac{\mu _{11}}{4}z_{3}+\frac{\mu _{22}}{4%
}z_{4}-\frac{\mu _{12}}{2}z_{5}-\frac{1}{p+1}z_{6}=\theta , \\
z_{1}+z_{2}+\mu _{11}z_{3}+\mu _{22}z_{4}-2\mu _{12}z_{5}-z_{6}=0, \\
\frac{1}{2}z_{1}+\frac{3}{2}z_{2}+\frac{5\mu _{11}}{4}z_{3}+\frac{5\mu _{22}%
}{4}z_{4}-\frac{5\mu _{12}}{2}z_{5}-\frac{3}{p+1}z_{6}=0, \\
z_{i}>0\text{ for }i=1,2,3,4,5,6.%
\end{array}%
\right.  \label{4-10}
\end{equation}%
Furthermore, using $\left( \ref{2-8-1}\right)$ and $\left( \ref{2-8-2}\right) $, we
have a constraint
\begin{equation}
z_{3}^{2}+z_{4}^{2}\leq \frac{256}{27\pi ^{2}\lambda ^{3}}z_{2}^{3}z_{1}.
\label{2-12}
\end{equation}%
Next, we want to show that%
\begin{equation}
-\left( p-1\right) \left( z_{1}+z_{2}\right) +\left( 3-p\right) \mu
_{11}z_{3}+\left( 3-p\right) \mu _{22}z_{4}-2\left( 3-p\right) \mu
_{12}z_{5}<0, \text{ for all }\mu _{ii}\in \left( 0,\overline{\Lambda }%
_{0}\right) .  \label{c1}
\end{equation}%
The general solution of the linear system $\left( \ref{4-10}\right) $
is given by
\begin{equation}
\left[
\begin{array}{c}
z_{1} \\
z_{2} \\
z_{3} \\
z_{4} \\
z_{5} \\
z_{6}%
\end{array}%
\right] =\theta \left[
\begin{array}{c}
1 \\
3 \\
0 \\
0 \\
\frac{2}{\mu _{12}} \\
0%
\end{array}%
\right] +s\left[
\begin{array}{c}
0 \\
0 \\
0 \\
\frac{1}{\mu _{22}} \\
\frac{1}{2\mu _{12}} \\
0%
\end{array}%
\right] +t\left[
\begin{array}{c}
0 \\
0 \\
\frac{1}{\mu _{11}} \\
0 \\
\frac{1}{2\mu _{12}} \\
0%
\end{array}%
\right] +w\left[
\begin{array}{c}
p-1 \\
-2(p-2) \\
0 \\
0 \\
\frac{-(p-1)}{\mu _{12}} \\
p+1%
\end{array}%
\right] ,  \label{4-12}
\end{equation}%
where $s,\ t,\ w\in \mathbb{R}$. Since $2<p<3,\, \mu _{ij}>0$ for $i,j=1,2$ and,
from (\ref{4-10}), $z_{i}>0$ for $i=1,\cdots ,6$, this is equivalent to
\begin{equation}
\left\{
\begin{array}{l}
3\theta -2w(p-2)>0, \\
4\theta +s+t-2w(p-1)>0, \\
s>0,\ \ t>0,\ \ w>0.\
\end{array}%
\right.  \label{4-13}
\end{equation}%
To verify $\left( \ref{c1}\right) $ is true, we substitute $\left( \ref{4-12}%
\right) $ into $\left( \ref{c1}\right) $ to give
\begin{equation}
-(p-1)(4\theta +(3-p)w)+(3-p)\left[ t+2(p-1)w+s-2(2\theta +t/2+s/2)\right]
<0.  \label{4-14-1}
\end{equation}%
This leads to
\begin{equation}
w<\frac{8}{(p-1)(3-p)}\theta .  \label{4-14}
\end{equation}%
From $\left( \ref{4-13}\right) $, we have
\begin{equation}
w<\frac{3}{2p-4}\theta .  \label{4-15}
\end{equation}%
This is a necessary condition for guaranteeing $z_{2}>0$. Note that,
comparing $\left( \ref{4-14}\right) $ and $\left( \ref{4-15}\right) $, we
obtain that if $\frac{\sqrt{73}-2}{3}\leq p<3$ and $\left( \ref{4-15}\right)
$ holds, then the condition $\left( \ref{4-14-1}\right) $ is satisfied, meaning that
$\left( \ref{c1}\right) $ is true for all $\mu _{ij}>0.$ Let $%
c_{\lambda }=\frac{256}{27\pi ^{2}\lambda ^{3}}.$ Applying the constraint $\left( %
\ref{2-12}\right) $, we substitute $\left( \ref{2-9-1}\right)$ and $\left( \ref%
{4-12}\right) $ into $\left( \ref{2-12}\right) $ to obtain
\begin{equation}
\frac{1}{\mu _{11}^{2}}t^{2}+\frac{1}{\mu _{22}^{2}}s^{2}\leq c_{\lambda
}\left( \frac{5-p}{p-1}\theta \right) ^{3}\left( \theta +(p-1)w\right) .
\label{4-16}
\end{equation}%
From the second inequality of $\left( \ref{4-13}\right) $, we have
\begin{equation}
w<\frac{2}{(p-1)}\theta +\frac{1}{2(p-1)}(s+t).  \label{4-17}
\end{equation}%
This is then used in $\left( \ref{4-16}\right) $ to give
\begin{equation*}
\frac{1}{\mu _{11}^{2}}t^{2}+\frac{1}{\mu _{22}^{2}}s^{2}<c_{\lambda }\left(
\frac{5-p}{p-1}\theta \right) ^{3}\left( 3\theta +\frac{1}{2}(s+t)\right) .
\end{equation*}%
Moreover, if $\mu _{ii}\in (0,\overline{\Lambda }_{0})$, then
\begin{equation}
t^{2}+s^{2}-\frac{c_{\lambda }}{2}\left( \frac{5-p}{p-1}\right) ^{3}%
\overline{\Lambda }_{0}^{2}\theta ^{3}(s+t)-3c_{\lambda }\left( \frac{5-p}{%
p-1}\right) ^{3}\overline{\Lambda }_{0}^{2}\theta ^{4}<0.  \label{4-19}
\end{equation}%
Since $t^{2}+s^{2}\geq \frac{1}{2}\left( s+t\right) ^{2}$, we have
\begin{equation}
\left( s+t\right) ^{2}-\overline{\Lambda }_{0}^{2}c_{\lambda }\left( \frac{%
5-p}{p-1}\right) ^{3}\theta ^{3}(s+t)-6\overline{\Lambda }_{0}^{2}c_{\lambda
}\left( \frac{5-p}{p-1}\right) ^{3}\theta ^{4}<0.  \label{4-20}
\end{equation}%
Substituting $\overline{\Lambda }_{0}>0$ and $c_{\lambda }=\frac{256}{27\pi
^{2}\lambda ^{3}}$ into $\left( \ref{4-20}\right) $, and since $\theta <D_{0},$ we
can conclude that
\begin{equation*}
l^{2}-\frac{(p+1)^{2}}{(3-p)(5-p)}\theta l-\frac{6(p+1)^{2}}{(3-p)(5-p)}%
\theta ^{2}<0,\text{ where }l=s+t.
\end{equation*}%
Since $l$ is positive, we must have
\begin{equation}
0<l<l_{0},  \label{4-21}
\end{equation}%
where
\begin{align*}
l_{0}& =\frac{(p+1)^{2}\theta }{(3-p)(5-p)}+\theta \sqrt{\frac{(p+1)^{4}}{%
(3-p)^{2}(5-p)^{2}}+\frac{24(p+1)^{2}}{(3-p)(5-p)}} \\
& =\frac{(p+1)^{2}\theta }{(3-p)(5-p)}+\frac{(p+1)\left( 19-5p\right)
\theta }{(3-p)(5-p)} \\
& =\frac{4(p+1)\theta }{3-p}.
\end{align*}%
Thus, for $\mu _{ii}\in (0,\overline{\Lambda }_{0})$ and $2<p<\frac{\sqrt{73}%
-2}{3},$  it follows from $\left( \ref{4-21}\right) $ that
\begin{equation*}
0<s+t<\frac{4(p+1)\theta }{3-p},
\end{equation*}%
implying
\begin{equation*}
w<\frac{2\theta }{(p-1)}+\frac{2(p+1)\theta }{(p-1)(3-p)}=\frac{8\theta }{%
(p-1)(3-p)}.
\end{equation*}%
This shows that the inequality $\left( \ref{c1}\right) $ holds. Therefore, $%
\mathbb{A}\subset \mathbf{M}^{-}.$

\vspace{2mm}
\noindent $\left( ii\right) $ Since $p=2,$ by $\left( \ref{2-9}\right) ,$%
\begin{equation}
\theta =\frac{p-1}{5-p}\left( \int_{\mathbb{R}^{3}}\lambda u_{0}^{2}dx+\int_{%
\mathbb{R}^{3}}\lambda v_{0}^{2}dx\right) >0  \label{2-9-2}.
\end{equation}%
Then, by $\left( \ref{2-9-2}\right) $ and using the argument similar to that in part $%
\left( i\right) ,$ we obtain that%
\begin{equation*}
-\left( p-1\right) \left( z_{1}+z_{2}\right) +\left( 3-p\right) \mu
_{11}z_{3}+\left( 3-p\right) \mu _{22}z_{4}-2\left( 3-p\right) \mu
_{12}z_{5}<0,\,\text{ for all }\mu _{ii}\in \left( 0,\overline{\Lambda }%
_{0}\right) .
\end{equation*}%
This completes the proof.
\end{proof}

\begin{remark}
Clearly, $\overline{\Lambda }_{0}>\Lambda _{0},\,$ for all $2\leq p\leq \frac{%
\sqrt{73}-2}{3}.$
\end{remark}

\textbf{We are now ready to prove Theorem \ref{t2}:} $\left( i\right) $ For $%
\mu _{ii}\in \left( 0,\Lambda _{0}\right) .$ Since $\overline{\Lambda }%
_{0}>\Lambda _{0},$ by Theorem \ref{t1}, System $(E)$ has a vectorial
solution $\left( u_{0}^{\left( 1\right) },v_{0}^{\left( 1\right) }\right)
\in \mathbf{M}^{\left( 1\right) }$ with%
\begin{equation*}
J\left( u_{0}^{\left( 1\right) },v_{0}^{\left( 1\right) }\right) =\alpha
^{-}=\inf_{u\in \mathbf{M}^{-}}J\left( u,v\right) <\beta _{\mu
_{ii}}^{\infty }.
\end{equation*}%
Since $\overline{\Lambda }_{0}>\Lambda _{0},$ $\alpha ^{-}<D_{0}$ for $\mu
_{ii}\in \left( 0,\Lambda _{0}\right) $ and $\mu _{11}\mu _{22}-\mu
_{12}^{2}\geq 0,$ by Proposition \ref{t6}, we can conclude that
\begin{equation*}
J\left( u_{0}^{\left( 1\right) },v_{0}^{\left( 1\right) }\right) =\alpha
^{-}=\inf_{u\in \mathbb{A}}J\left( u,v\right) ,
\end{equation*}%
which implies that $\left( u_{0}^{\left( 1\right) },v_{0}^{\left( 1\right)
}\right) $ is a positive ground state solution of System $(E).$ \newline
$\left( ii\right) $ The proof is similar to that of part $(i)$ and is therefore omitted here.

\section{Nonexistence of nontrivial solutions}

\bigskip

\textbf{We are now ready to prove Theorem \ref{T6}. }Suppose that $\left(
u,v\right) \in H$ is a notrivial solution of System $(E).$ Then
\begin{equation}
\left\Vert \left( u,v\right) \right\Vert _{H}^{2}+\int_{\mathbb{R}^{3}}\mu
_{11}\phi _{_{u}}u^{2}+\mu _{22}\phi _{_{v}}v^{2}-2\mu _{12}\phi
_{_{v}}u^{2}dx-\frac{1}{2\pi }\int_{\mathbb{R}^{3}}\int_{0}^{2\pi
}\left\vert u+e^{i\theta }v\right\vert ^{p+1}d\theta dx=0  \label{6-1-1}
\end{equation}%
or%
\begin{eqnarray}
0 &=&\int_{\mathbb{R}^{3}}\left\vert \nabla u\right\vert ^{2}+\left\vert
\nabla v\right\vert ^{2}+\lambda _{1}u^{2}+\lambda _{2}v^{2}dx+\int_{\mathbb{%
R}^{3}}\mu _{11}\phi _{_{u}}u^{2}+\mu _{22}\phi _{_{v}}v^{2}-2\mu _{12}\phi
_{_{v}}u^{2}dx  \notag \\
&&-\frac{1}{2\pi }\int_{\mathbb{R}^{3}}\int_{0}^{2\pi }\left( u^{2}+2uv\cos
\theta +v^{2}\right) ^{\frac{p+1}{2}}d\theta dx  \label{6-1-2}
\end{eqnarray}%
We note that%
\begin{equation*}
\frac{1}{2\pi }\int_{0}^{2\pi }\left( u^{2}+2uv\cos \theta +v^{2}\right) ^{%
\frac{p+1}{2}}d\theta \leq \left( \left\vert u\right\vert +\left\vert
v\right\vert \right) ^{p+1}\leq 2^{p}\left( \left\vert u\right\vert
^{p+1}+\left\vert v\right\vert ^{p+1}\right) .
\end{equation*}%
By the definition of $\phi _{w},$ we have that
\begin{equation*}
\int_{\mathbb{R}^{3}}\phi _{w}u^{2}dx=\int_{\mathbb{R}^{3}}\left\vert \nabla
\phi _{w}\right\vert ^{2}dx\text{ for }w=u,v.
\end{equation*}%
Moreover, by inequality $\left( 13\right) $ in \cite{Jin},%
\begin{equation*}
\mu _{11}\left( \mu _{11}\left\vert \nabla \phi _{u}\right\vert ^{2}+\mu
_{22}\left\vert \nabla \phi _{v}\right\vert ^{2}-2\mu _{12}\nabla \phi
_{u}\cdot \nabla \phi _{v}\right) \geq \left( \mu _{11}\mu _{22}-\mu
_{12}^{2}\right) \left\vert \nabla \phi _{v}\right\vert ^{2}
\end{equation*}%
and%
\begin{equation*}
\mu _{22}\left( \mu _{11}\left\vert \nabla \phi _{u}\right\vert ^{2}+\mu
_{22}\left\vert \nabla \phi _{v}\right\vert ^{2}-2\mu _{12}\nabla \phi
_{u}\cdot \nabla \phi _{v}\right) \geq \left( \mu _{11}\mu _{22}-2\mu
_{12}^{2}\right) \left\vert \nabla \phi _{u}\right\vert ^{2},
\end{equation*}%
this implies that
\begin{equation}
\mu _{11}\left\vert \nabla \phi _{u}\right\vert ^{2}+\mu _{22}\left\vert
\nabla \phi _{v}\right\vert ^{2}-2\mu _{12}\nabla \phi _{u}\cdot \nabla \phi
_{v}\geq \frac{\mu _{11}\mu _{22}-\mu _{12}^{2}}{\mu _{11}+\mu _{22}}\left(
\left\vert \nabla \phi _{u}\right\vert ^{2}+\left\vert \nabla \phi
_{v}\right\vert ^{2}\right) .  \label{6-2-1}
\end{equation}%
On the other hand, we deduce that%
\begin{eqnarray}
2\left( \frac{\mu _{11}\mu _{22}-\mu _{12}^{2}}{\mu _{11}+\mu _{22}}\right)
^{1/2}\int_{\mathbb{R}^{3}}\left\vert w\right\vert ^{3}dx &=&2\left( \frac{%
\mu _{11}\mu _{22}-\mu _{12}^{2}}{\mu _{11}+\mu _{22}}\right) ^{1/2}\int_{%
\mathbb{R}^{3}}\left( -\Delta \phi _{_{w}}\right) \left\vert w\right\vert dx
\notag \\
&=&2\left( \frac{\mu _{11}\mu _{22}-2\mu _{12}^{2}}{\mu _{11}+\mu _{22}}%
\right) ^{1/2}\int_{\mathbb{R}^{3}}\nabla \phi _{_{w}}\cdot \nabla
\left\vert w\right\vert dx  \notag \\
&\leq &\int_{\mathbb{R}^{3}}\left\vert \nabla w\right\vert ^{2}dx+\frac{\mu
_{11}\mu _{22}-\mu _{12}^{2}}{\mu _{11}+\mu _{22}}\int_{\mathbb{R}%
^{3}}\left\vert \nabla \phi _{w}\right\vert ^{2}dx  \notag \\
&=&\int_{\mathbb{R}^{3}}\left\vert \nabla w\right\vert ^{2}dx+\frac{\mu
_{11}\mu _{22}-\mu _{12}^{2}}{\mu _{11}+\mu _{22}}\int_{\mathbb{R}^{3}}\phi
_{_{w}}w^{2}dx.  \label{6-2-2}
\end{eqnarray}%
Thus, by $\left( \ref{6-1-2}\right) -\left( \ref{6-2-2}\right) ,$ we can
conclude that%
\begin{eqnarray}
0 &=&\int_{\mathbb{R}^{3}}\left\vert \nabla u\right\vert ^{2}+\left\vert
\nabla v\right\vert ^{2}+\lambda u^{2}+\lambda v^{2}dx+\int_{\mathbb{R}%
^{3}}\mu _{11}\phi _{_{u}}u^{2}+\mu _{22}\phi _{_{v}}v^{2}-2\mu _{12}\phi
_{_{v}}u^{2}dx  \notag \\
&&-\frac{1}{2\pi }\int_{\mathbb{R}^{3}}\int_{0}^{2\pi }\left( u^{2}+2uv\cos
\theta +v^{2}\right) ^{\frac{p+1}{2}}d\theta dx  \notag \\
&\geq &\int_{\mathbb{R}^{3}}\left\vert \nabla u\right\vert ^{2}+\left\vert
\nabla v\right\vert ^{2}+\lambda u^{2}+\lambda v^{2}dx+\frac{\mu _{11}\mu
_{22}-\mu _{12}^{2}}{\mu _{11}+\mu _{22}}\left( \int_{\mathbb{R}^{3}}\phi
_{_{u}}u^{2}dx+\int_{\mathbb{R}^{3}}\phi _{_{v}}v^{2}dx\right)   \notag \\
&&-2^{p}\int_{\mathbb{R}^{3}}\left\vert u\right\vert ^{p+1}dx-2^{p}\int_{%
\mathbb{R}^{3}}\left\vert v\right\vert ^{p+1}dx  \notag \\
&\geq &\int_{\mathbb{R}^{3}}u^{2}\left( \lambda -2^{p}\left\vert
u\right\vert ^{p-1}+2\left( \frac{\mu _{11}\mu _{22}-\mu _{12}^{2}}{\mu
_{11}+\mu _{22}}\right) ^{1/2}\left\vert u\right\vert \right) dx  \notag \\
&&+\int_{\mathbb{R}^{3}}v^{2}\left( \lambda -2^{p}\left\vert v\right\vert
^{p-1}+2\left( \frac{\mu _{11}\mu _{22}-\mu _{12}^{2}}{\mu _{11}+\mu _{22}}%
\right) ^{1/2}\left\vert v\right\vert \right) dx.  \label{6-3}
\end{eqnarray}%
If $1<p<2,$ then by Lemma \ref{L2-0} and $\left( \ref{6-3}\right) $, we can
conclude that $u=v\equiv 0$ for all $\lambda ,\mu _{ij}>0$ with
\begin{equation*}
\frac{\mu _{11}\mu _{22}-\mu _{12}^{2}}{\mu _{11}+\mu _{22}}>\frac{\left(
p-1\right) ^{2}}{4}\left[ \frac{2^{p}\left( 2-p\right) ^{2-p}}{\lambda ^{2-p}%
}\right] ^{2/\left( p-1\right) }.
\end{equation*}%
If $p=2,$ then by $\left( \ref{6-3}\right) $, we can conclude that $%
u=v\equiv 0$ for all $\lambda ,\mu _{ij}>0$ with%
\begin{equation*}
\frac{\mu _{11}\mu _{22}-\mu _{12}^{2}}{\mu _{11}+\mu _{22}}>4.
\end{equation*}%
This completes the proof.

\section{Existence of two positive solutions}

Following the idea in \cite{R1} and inequality $\left( \ref{6-2-1}\right) $,
we study the existence of two positive solutions of System $(E)$ for $1<p<2$
and $\mu _{11}\mu _{22}-\mu _{12}^{2}>0.$ Then we have the following result.

\begin{proposition}
\label{p6}Suppose that $1<p<2$ and $\mu _{ij}>0.$ If $\mu _{11}\mu _{22}-\mu
_{12}^{2}>0,$ then we have\newline
$\left( i\right) $ $J$ is bounded from below and coercive in $H_{r};$\newline
$\left( ii\right) $ $J$ satisfies $\left( PS\right) $ condition in $H_{r}.$
\end{proposition}

\begin{proof}
$\left( i\right) $ Since $\mu _{ii}>0$ and $\mu _{11}\mu _{22}-\mu
_{12}^{2}>0,$
\begin{eqnarray}
\frac{1}{2\sqrt{2}}\left( \frac{\mu _{11}\mu _{22}-\mu _{12}^{2}}{\mu
_{11}+\mu _{22}}\right) ^{1/2}\int_{\mathbb{R}^{3}}\left\vert w\right\vert
^{3}dx &=&\frac{1}{2\sqrt{2}}\left( \frac{\mu _{11}\mu _{22}-\mu _{12}^{2}}{%
\mu _{11}+\mu _{22}}\right) ^{1/2}\int_{\mathbb{R}^{3}}\left( -\Delta \phi
_{_{w}}\right) \left\vert w\right\vert dx  \notag \\
&=&\frac{1}{2\sqrt{2}}\left( \frac{\mu _{11}\mu _{22}-\mu _{12}^{2}}{\mu
_{11}+\mu _{22}}\right) ^{1/2}\int_{\mathbb{R}^{3}}\nabla \phi _{_{w}}\cdot
\nabla \left\vert w\right\vert dx  \notag \\
&\leq &\frac{1}{4}\int_{\mathbb{R}^{3}}\left\vert \nabla w\right\vert ^{2}dx+%
\frac{\mu _{11}\mu _{22}-\mu _{12}^{2}}{8\left( \mu _{11}+\mu _{22}\right) }%
\int_{\mathbb{R}^{3}}\left\vert \nabla \phi _{w}\right\vert ^{2}dx  \notag \\
&=&\frac{1}{4}\int_{\mathbb{R}^{3}}\left\vert \nabla w\right\vert ^{2}dx+%
\frac{\mu _{11}\mu _{22}-\mu _{12}^{2}}{8\left( \mu _{11}+\mu _{22}\right) }%
\int_{\mathbb{R}^{3}}\phi _{_{w}}w^{2}dx  \label{7-1}
\end{eqnarray}%
for $w=u,v.$ We note that%
\begin{equation}
\frac{1}{2\pi }\int_{0}^{2\pi }\left( u^{2}+2uv\cos \theta +v^{2}\right) ^{%
\frac{p+1}{2}}d\theta \leq \left( \left\vert u\right\vert +\left\vert
v\right\vert \right) ^{p+1}\leq 2^{p}\left( \left\vert u\right\vert
^{p+1}+\left\vert v\right\vert ^{p+1}\right) .  \label{7-2}
\end{equation}%
Then by inequalities $\left( \ref{6-2-1}\right) ,\left( \ref{7-1}\right) $
and $\left( \ref{7-2}\right) ,$%
\begin{eqnarray}
J\left( u,v\right)  &=&\frac{1}{2}\left\Vert \left( u,v\right) \right\Vert
_{H}^{2}+\frac{1}{4}\int_{\mathbb{R}^{3}}\mu _{11}\phi _{_{u}}u^{2}+\mu
_{22}\phi _{_{v}}v^{2}-2\mu _{12}\phi _{_{v}}u^{2}dx  \notag \\
&&-\frac{1}{2\pi \left( p+1\right) }\int_{\mathbb{R}^{3}}\int_{0}^{2\pi
}\left\vert u+e^{i\theta }v\right\vert ^{p+1}d\theta dx  \notag \\
&\geq &\frac{1}{4}\int_{\mathbb{R}^{3}}\left\vert \nabla u\right\vert
^{2}+\lambda u^{2}dx+\frac{\mu _{11}\mu _{22}-\mu _{12}^{2}}{8\left( \mu
_{11}+\mu _{22}\right) }\int_{\mathbb{R}^{3}}\phi _{_{u}}u^{2}dx  \notag \\
&&+\frac{1}{4}\int_{\mathbb{R}^{3}}\lambda u^{2}+\frac{1}{2\sqrt{2}}\left(
\frac{\mu _{11}\mu _{22}-\mu _{12}^{2}}{\mu _{11}+\mu _{22}}\right)
^{1/2}\left\vert u\right\vert ^{3}-\frac{2^{p}}{p+1}\left\vert u\right\vert
^{p+1}dx  \notag \\
&&+\frac{1}{4}\int_{\mathbb{R}^{3}}\left\vert \nabla v\right\vert
^{2}+\lambda v^{2}dx+\frac{\mu _{11}\mu _{22}-\mu _{12}^{2}}{8\left( \mu
_{11}+\mu _{22}\right) }\int_{\mathbb{R}^{3}}\phi _{_{v}}v^{2}dx  \notag \\
&&+\frac{1}{4}\int_{\mathbb{R}^{3}}\lambda v^{2}+\frac{1}{2\sqrt{2}}\left(
\frac{\mu _{11}\mu _{22}-\mu _{12}^{2}}{\mu _{11}+\mu _{22}}\right)
^{1/2}\left\vert v\right\vert ^{3}-\frac{2^{p}}{p+1}\left\vert v\right\vert
^{p+1}dx.  \label{7-3}
\end{eqnarray}%
Thus, by $\left( \ref{7-3}\right) $ and applying the argument in Ruiz \cite[%
Theorem 4.3]{R1}, $J$ is coercive on $H_{r}$ and there exists $M>0$ such that%
\begin{equation*}
\inf_{\left( u,v\right) \in H_{r}}J(u,v)\geq -M.
\end{equation*}%
$\left( ii\right) $ By \cite[Proposition 6.1 $\left( ii\right) $]{Jin}.
\end{proof}

Assume that $w_{r,\mu }^{\left( 1\right) }$ and $w_{r,\mu }^{\left( 2\right)
}$ are positive radial solutions of Equation $\left( SP_{\mu }\right) $ as
in Theorem \ref{T5}, that is%
\begin{equation*}
I_{\mu }\left( w_{r,\mu }^{\left( 1\right) }\right) =\beta _{r,\mu }^{\left(
1\right) }:=\inf_{u\in \mathbf{N}_{\mu }^{-}\cap H_{r}^{1}\left( \mathbb{R}%
^{3}\right) }I_{\mu }\left( u\right) >0
\end{equation*}%
and
\begin{equation*}
I_{\mu }\left( w_{r,\mu }^{\left( 2\right) }\right) =\beta _{r,\mu }^{\left(
2\right) }:=\inf_{u\in \mathbf{N}_{\mu }^{+}\cap H_{r}^{1}\left( \mathbb{R}%
^{3}\right) }I_{\mu }\left( u\right) =\inf_{u\in H_{r}^{1}\left( \mathbb{R}%
^{3}\right) }I_{\mu }\left( u\right) <0.
\end{equation*}%
Then we have the following results.

\begin{lemma}
\label{L6-2}Suppose that $1<p<2$ and $\mu _{ij}>0.$ If $0<\mu _{11}<\Lambda
_{0}$ and $\mu _{11}\mu _{22}-\mu _{12}^{2}>0,$ then we have \newline
$\left( i\right) $ $J\left( \sqrt{s_{\min }}w_{r,\mu _{11}}^{\left( 2\right)
},\sqrt{1-s_{\min }}w_{r,\mu _{11}}^{\left( 2\right) }\right) <I_{\mu
}\left( w_{r,\mu _{11}}^{\left( 2\right) }\right) =\beta _{r,\mu
_{11}}^{\left( 2\right) }<0;$\newline
$\left( ii\right) $ Let $\left( u_{0},v_{0}\right) $ be a critical point of $%
J $ on $\mathbf{M}^{+}\cap H_{r}.$ Then we have $J\left( u_{0},v_{0}\right)
\geq \beta _{r,\mu }^{\left( 2\right) }$ if either $u_{0}=0$ or $v_{0}=0.$
\end{lemma}

\begin{proof}
$\left( i\right) $ Since%
\begin{eqnarray*}
&&J\left( \sqrt{s_{\min }}w_{r,\mu _{11}}^{\left( 2\right) },\sqrt{1-s_{\min
}}w_{r,\mu _{11}}^{\left( 2\right) }\right) \\
&=&\frac{1}{2}\left\Vert w_{r,\mu _{11}}^{\left( 2\right) }\right\Vert
_{H^{1}}^{2}+\frac{\mu _{11}\mu _{22}-\mu _{12}^{2}}{4\left( \mu _{11}+\mu
_{22}+2\mu _{12}\right) }\int_{\mathbb{R}^{3}}\phi _{_{w_{r,\mu
_{11}}^{\left( 2\right) }}}\left\vert w_{r,\mu _{11}}^{\left( 2\right)
}\right\vert ^{2}dx \\
&&-\frac{1}{2\pi \left( p+1\right) }\int_{\mathbb{R}^{3}}\int_{0}^{2\pi
}\left( 1+2\sqrt{s_{\min }\left( 1-s_{\min }\right) }\cos \theta \right)
^{\left( p+1\right) /2}\left\vert w_{r,\mu _{11}}^{\left( 2\right)
}\right\vert ^{p+1}d\theta dx,
\end{eqnarray*}%
and%
\begin{eqnarray*}
\frac{1}{2\pi }\int_{0}^{2\pi }\left( 1+2\sqrt{s_{\min }\left( 1-s_{\min
}\right) }\cos \theta \right) ^{\left( p+1\right) /2}d\theta &>&\left( \frac{%
1}{2\pi }\int_{0}^{2\pi }1+2\sqrt{s_{\min }\left( 1-s_{\min }\right) }\cos
\theta d\theta \right) ^{\left( p+1\right) /2} \\
&=&1,
\end{eqnarray*}%
we have%
\begin{eqnarray*}
&&J\left( \sqrt{s_{\min }}w_{r,\mu _{11}}^{\left( 2\right) },\sqrt{1-s_{\min
}}w_{r,\mu _{11}}^{\left( 2\right) }\right) \\
&<&\frac{1}{2}\left\Vert w_{r,\mu _{11}}^{\left( 2\right) }\right\Vert
_{H^{1}}^{2}+\frac{\mu _{11}}{4}\int_{\mathbb{R}^{3}}\phi _{_{w_{r,\mu
_{11}}^{\left( 2\right) }}}\left\vert w_{r,\mu _{11}}^{\left( 2\right)
}\right\vert ^{2}dx-\frac{1}{p+1}\int_{\mathbb{R}^{3}}\left\vert w_{r,\mu
_{11}}^{\left( 2\right) }\right\vert ^{p+1}d\theta dx \\
&=&I_{\mu }\left( w_{r,\mu _{11}}^{\left( 2\right) }\right) =\beta _{r,\mu
_{11}}^{\left( 2\right) }.
\end{eqnarray*}%
$\left( ii\right) $ Without loss of generality, we may assume that $v_{0}=0.$
Then
\begin{equation*}
J\left( u_{0},0\right) =I_{\mu _{11}}\left( u_{0}\right) =\frac{1}{2}%
\left\Vert u_{0}\right\Vert _{H^{1}}^{2}+\frac{\mu _{11}}{4}\int_{\mathbb{R}%
^{3}}\phi _{u_{0}}u_{0}^{2}dx-\frac{1}{p+1}\int_{\mathbb{R}^{3}}\left\vert
u_{0}\right\vert ^{p+1}dx
\end{equation*}%
and%
\begin{equation*}
h_{\left( u,0\right) }^{\prime \prime }\left( 1\right) =f_{u}^{\prime \prime
}\left( 1\right) =-2\left\Vert u_{0}\right\Vert _{H^{1}}^{2}+\left(
3-p\right) \int_{\mathbb{R}^{3}}\left\vert u_{0}\right\vert ^{p+1}dx>0,
\end{equation*}%
implying that $u_{0}\in \mathbf{N}_{\mu _{11}}^{+}\cap H_{r}^{1}\left(
\mathbb{R}^{3}\right) .$ Thus $J\left( u_{0},0\right) =$ $I_{\mu _{11}}\left(
u_{0}\right) \geq \beta _{r,\mu }^{\left( 2\right) }.$ This completes the
proof.
\end{proof}

\bigskip

\textbf{We are now ready to prove Theorem \ref{t3}. }By Proposition \ref{p6}
$\left( i\right) $ and Lemma \ref{L6-2}, we can apply the Ekeland
variational principle \cite{E} and Palais criticality principle \cite{P} to
obtain that there exists a sequence $\{\left( u_{n}^{\left( 2\right)
},v_{n}^{\left( 2\right) }\right) \}\subset H_{r}\diagdown \left\{ \left(
0,0\right) \right\} $ such that
\begin{equation*}
J(u_{n}^{\left( 2\right) },v_{n}^{\left( 2\right) })=\inf_{\left( u,v\right)
\in H_{r}}J\left( u,v\right) +o(1)\text{ and }J(u_{n}^{\left( 2\right)
},v_{n}^{\left( 2\right) })=o(1)\text{ in }H^{-1}
\end{equation*}%
and%
\begin{equation*}
\inf_{\left( u,v\right) \in H_{r}}J\left( u,v\right) <\beta _{r,\mu
_{11}}^{\left( 2\right) }<0,
\end{equation*}%
for $i=1,2.$ Then by Proposition \ref{p6} $\left( ii\right) ,$ there exists a vectorial solution $%
\left( u_{0}^{\left( 2\right) },v_{0}^{\left( 2\right) }\right) \in
H_{r}\diagdown \left\{ \left( 0,0\right) \right\} $ such that
\begin{eqnarray*}
\left( u_{n}^{\left( 2\right) },v_{n}^{\left( 2\right) }\right) &\rightarrow &\left(
u_{0}^{\left( 2\right)},v_{0}^{\left( 2\right)}\right) \text{ strongly in }%
H_{r}; \\
J(u_{0}^{\left( 2\right) },v_{0}^{\left( 2\right) }) &=&\inf_{\left(
u,v\right) \in H_{r}}J\left( u,v\right) .
\end{eqnarray*}%
Since
\begin{equation*}
J(\left\vert u_{0}^{\left( 2\right) }\right\vert ,\left\vert v_{0}^{\left(
2\right) }\right\vert )=J(u_{0}^{\left( 2\right) },v_{0}^{\left( 2\right)
})=\inf_{\left( u,v\right) \in H_{r}}J\left( u,v\right) \text{ for }i=1,2,
\end{equation*}%
we may assume that $\left( u_{0}^{\left( 2\right) },v_{0}^{\left( 2\right)
}\right) $ are positive critical points of $J$ on $H_{r}$. Moreover, by
Lemma \ref{L6-2} and $\inf_{\left( u,v\right) \in H_{r}}J\left( u,v\right)
<\beta _{r,\mu _{11}}^{\left( 2\right) }<0$ we have $u_{0}^{\left( 2\right)
}\neq 0$ and $v_{0}^{\left( 2\right) }\neq 0.$ Combining this result with
Theorem \ref{t1}, we conclude that System $\left( E\right) $ has two
positive solutions $\left( u_{0}^{\left( 1\right) },v_{0}^{\left( 1\right)
}\right) $ and $\left( u_{0}^{\left( 2\right) },v_{0}^{\left( 2\right)
}\right) $ such that
\begin{equation*}
J(u_{0}^{\left( 2\right) },v_{0}^{\left( 2\right) })<0<\frac{p-1}{4\left(
p+1\right) }C_{\mu _{12}}^{2}<\alpha ^{-}=J(u_{0}^{\left( 1\right)
},v_{0}^{\left( 1\right) }).
\end{equation*}%
This completes the proof.

\section*{Acknowledgments}

T.F. Wu was supported by the National Science and Technology Council, Taiwan
(Grant No. 112-2115-M-390-001-MY3).

\end{document}